\documentclass{amsart}

\usepackage[a4paper,hmargin=3cm,vmargin=4cm]{geometry}
\usepackage{amsfonts,amssymb,amscd,amstext}
\usepackage{graphicx}
\usepackage[dvips]{epsfig}
\usepackage[colorlinks=true,linkcolor=blue,citecolor=red]{hyperref}

\renewcommand{\leq}{\leqslant}
\renewcommand{\geq}{\geqslant}
\newcommand{\ptl}{\partial}
\newcommand{\rr}{{\mathbb{R}}}
\newcommand{\la}{\lambda}
\newcommand{\h}{\mathcal{H}}
\newcommand{\nn}{\mathbb{N}}
\newcommand{\sub}{\subset}
\newcommand{\subeq}{\subseteq}
\newcommand{\escpr}[1]{\big<#1\big>}
\newcommand{\Sg}{\Sigma} \newcommand{\sg}{\sigma}
\newcommand{\Om}{\Omega}

\newcommand{\eps}{\varepsilon}
\newcommand{\var}{\varphi}
\newcommand{\ga}{\gamma}
\newcommand{\Ga}{\Gamma}
\newcommand{\mnh}{|N_{h}|}
\newcommand{\nuh}{\nu_{h}}

\newcommand{\m}{\mathbb{N}(\kappa)}
\newcommand{\e}{\mathbb{M}(\kappa)}
\newcommand{\stres}{\mathbb{S}^3}
\newcommand{\mm}{\mathbb{M}}
\newcommand{\s}{\mathbb{S}}
\newcommand{\ele}{\mathcal{L}}
\newcommand{\cmula}{\mathcal{C}_{\lambda}}
\newcommand{\pla}{\mathcal{P}_\lambda(p)}

\newcommand{\sla}{\mathcal{S}_\la}

\DeclareMathOperator{\divv}{div}

\newtheorem{theorem}{Theorem}[section]
\newtheorem{proposition}[theorem]{Proposition}
\newtheorem{lemma}[theorem]{Lemma}
\newtheorem{corollary}[theorem]{Corollary}

\theoremstyle{definition}

\newtheorem{remark}[theorem]{Remark}
\newtheorem{remarks}[theorem]{Remarks}
\newtheorem{example}[theorem]{Example}

\theoremstyle{remark}

\numberwithin{equation}{section}

\begin{document}

\bibliographystyle{amsplain}

\title[Strongly stable surfaces in sub-Riemannian $3$-space forms]
{Strongly stable surfaces in sub-Riemannian $3$-space forms}

\author[A.~Hurtado]{Ana Hurtado}
\address{Departamento de Geometr\'{\i}a y Topolog\'{\i}a \\
Universidad de Granada \\ E--18071 Granada \\ Spain}
\email{ahurtado@ugr.es}

\author[C.~Rosales]{C\'esar Rosales}
\address{Departamento de Geometr\'{\i}a y Topolog\'{\i}a \\
Universidad de Granada \\ E--18071 Granada \\ Spain}
\email{crosales@ugr.es}

\date{October 11, 2016}

\thanks{The authors have been supported by Mineco-Feder grant MTM2013-48371-C2-1-P and Junta de Andaluc\'ia grant FQM-325. The second author was also supported by the grant PYR-2014-23 of the GENIL program of CEI BioTic GRANADA}



\begin{abstract}
A surface of constant mean curvature (CMC) equal to $H$ in a sub-Riemannian $3$-manifold is strongly stable if it minimizes the functional $\text{area}+2H\,\text{volume}$ up to second order. In this paper we obtain some criteria ensuring strong stability of surfaces in Sasakian $3$-manifolds. We also produce new examples of $C^1$ complete CMC surfaces with empty singular set in the sub-Riemannian $3$-space forms by studying those ones containing a vertical line. As a consequence, we are able to find complete strongly stable non-vertical surfaces with empty singular set in the sub-Riemannian hyperbolic $3$-space $\mm(-1)$. In relation to the Bernstein problem in $\mm(-1)$ we discover strongly stable $C^\infty$ entire minimal graphs in $\mm(-1)$ different from vertical planes. These examples are in clear contrast with the situation in the first Heisenberg group, where complete strongly stable surfaces with empty singular set are vertical planes. Finally, we analyze the strong stability of CMC surfaces of class $C^2$ and non-empty singular set in the sub-Riemannian $3$-space forms. When these surfaces have isolated singular points we deduce their strong stability even for variations moving the singular set.
\end{abstract}

\maketitle

\thispagestyle{empty}

\section{Introduction}
\label{sec:intro}
\setcounter{equation}{0}

Let $M$ be a Sasakian sub-Riemannian $3$-manifold (to be defined in Section~\ref{subsec:ssR3m}). From the first variation formulas, see for instance \cite[Sect.~4.1]{hr2}, a surface $\Sg$ in $M$ with $\ptl\Sg=\emptyset$ which is a critical point of the (sub-Riemannian) area $A$ for any variation preserving the associated volume $V$ has constant mean curvature $H$ in the sense of \eqref{eq:mc2}. From here, it is easy to deduce that $\Sg$ satisfies $(A+2HV)'(0)=0$ for any variation. Following standard terminology we will say that $\Sg$ is \emph{strongly stable} if, furthermore, we have $(A+2HV)''(0)\geq 0$ for any variation. In particular, it is clear that $A''(0)\geq 0$ under volume-preserving variations. In the minimal case ($H=0$) the strong stability is the classical condition that $A''(0)\geq 0$ for any variation.

In recent years constant mean curvature (CMC) surfaces, stability properties and Bernstein type problems have been extensively investigated in sub-Riemannian manifolds. The present paper aims to study these topics inside Sasakian $3$-manifolds, focusing on the simplest and most symmetric ones: the \emph{space forms}, defined in Section~\ref{subsec:model} as complete Sasakian sub-Riemannian $3$-manifolds of constant Webster scalar curvature $\kappa$. In the simply connected case, a result of Tanno~\cite{tanno} establishes that the space form $\e$ is, up to isometries, the first Heisenberg group $\mathbb{H}^1$ for $\kappa=0$, the group of unit quaternions $\mathbb{S}^3\sub\rr^4$ for $\kappa=1$, and the universal cover of the special linear group $\text{SL}(2,\rr)$ for $\kappa=-1$. As the authors showed in \cite[Sect.~2.2]{hr2}, standard arguments in Riemannian geometry produce $3$-dimensional space forms with non-trivial topology.

The analysis of the stability condition requires an explicit expression for $(A+2HV)''(0)$. The second variation of the sub-Riemannian area has appeared in several contexts, see for instance \cite{chmy}, \cite{bscv}, \cite{dgn}, \cite{montefalcone}, \cite{mscv}, \cite{hrr}, \cite{ch2}, \cite{hp2}, \cite{galli}, \cite{montefalcone2} and \cite{galli-ritore3}. The computation of $A''(0)$ by differentiation under the integral sign in \eqref{eq:area} involves a technical problem since the deformation could move the \emph{singular set} $\Sg_0$, which consists of the points in $\Sg$ where the integrand $\mnh$ vanishes. In \cite{hr2} we solved this difficulty by assuming the employed deformations to be \emph{admissible}, which allows to apply the Leibniz's rule for differentiating under the integral sign. For a CMC surface $\Sg$ inside a a Sasakian sub-Riemannian $3$-manifold $M$, we were able to prove that
\begin{equation}
\label{eq:intro1}
(A+2HV)''(0)=\mathcal{Q}(w,w)+\int_\Sg\divv_\Sg G\,da,
\end{equation}
for any admissible variation under suitable integrability conditions. In the previous formula, $w$ denotes the normal component of the velocity vector field associated to the variation, $\divv_\Sg G$ is the divergence relative to $\Sg$ of a certain tangent vector field $G$ along $\Sg-\Sg_0$, see \cite[Eq.~(7.1)]{hr2}, and $\mathcal{Q}$ is the (sub-Riemannian) \emph{index form} of $\Sg$ defined in \eqref{eq:indexform}. For the particular case of variations supported on $\Sg-\Sg_0$ the divergence term vanishes, and we can use integration by parts to deduce
\[
(A+2HV)''(0)=-\int_{\Sg} w\,\mathcal{L}(w)\,da,
 \]
where $\mathcal{L}$ is the second order differential operator given in \eqref{eq:lu}. This operator plays in our setting the same role as the \emph{Jacobi operator} introduced by Barbosa, do Carmo and Eschenburg~\cite{bdce} for CMC hypersurfaces in Riemannian manifolds. By analogy with the Riemannian situation we define a (sub-Riemannian) \emph{Jacobi function} on $\Sg$ as a function $\psi\in C^2(\Sg)$ for which $\ele(\psi)=0$.

The second variation formula provides a bridge between the stability properties of $\Sg$ and the operator $\ele$. There is a vast literature exploring this connection in the Riemannian context. A relevant result in this line is a theorem of Fischer-Colbrie and Schoen \cite{fcs} asserting that a CMC hypersurface having a positive Jacobi function is strongly stable. Recently, Montefalcone~\cite{montefalcone2} has derived a sub-Riemannian counterpart of this theorem for minimal hypersurfaces with empty singular set in Carnot groups. In Section~\ref{sec:statstable} of the present paper we provide a similar stability criterion for a CMC surface $\Sg$ with $\Sg_0=\emptyset$ in a Sasakian sub-Riemannian $3$-manifold. Indeed, from the second variation formula \eqref{eq:intro1} and the expression in equation \eqref{eq:sorpasso} for the index form $\mathcal{Q}$, we show in Theorem~\ref{th:criterion} strong stability of $\Sg$ provided there is a nowhere vanishing function $\psi\in C^2(\Sg)$ such that $\psi\,\ele(\psi)\leq 0$. Since the normal component of the Reeb vector field $T$ in $M$ is always a Jacobi function (Lemma~\ref{lem:jacobi}) we deduce in Corollary~\ref{cor:criteria} that, if $\Sg$ does not contain vertical points (those where $T$ is tangent to $\Sg$), then $\Sg$ is strongly stable. Moreover, an immediate application of Theorem~\ref{th:criterion} yields strong stability of $\Sg$ whenever the function $\mnh$ satisfies $\ele(\mnh)\leq 0$.  It is worth mentioning that the relation between $\ele(\mnh)$ and the stability properties of minimal surfaces has been investigated in several works, see \cite{dgnp}, \cite{hrr}, \cite{dgnp-stable}, \cite{rosales}, \cite{galli}, \cite{galli2} and \cite{galli-ritore3}.

In Section~\ref{sec:helicoids} we study complete CMC surfaces with empty singular set in the model spaces $\e$ (though our construction and results can be extended to arbitrary sub-Riemannian $3$-space forms). The existence of smooth stable examples in $\e$ is very restrictive due to some rigidity results that we know summarize. In the Heisenberg group $\mm(0)$ any strongly stable minimal surface $\Sg$ with $\Sg_0=\emptyset$ is a vertical plane. By assuming $C^2$ regularity of $\Sg$ this was proved in \cite{hrr} and \cite{dgnp-stable} after some partial characterizations in \cite{bscv} and \cite{dgnp}. In the $C^1$ case the statement has been recently obtained by Galli and Ritor\'e \cite{galli-ritore3}. On the other hand, complete stable CMC surfaces of class $C^2$ and empty singular set in $\e$ were analyzed by the second author in \cite{rosales}. More precisely, if $\Sg$ is such a surface, then either $\kappa=0$ and $\Sg$ is a vertical plane, or $\kappa=-1$ and the mean curvature $H$ of $\Sg$ satisfies $H^2\leq 1$. Moreover, in the extremal case $H^2=1$, the surface $\Sg$ must be a vertical horocylinder. This result suggests that the stability condition is much less restrictive in the hyperbolic model $\mm(-1)$ whenever $0\leq H^2<1$. Motivated by this fact, the second author conjectured in \cite[Re.~6.10]{rosales} the existence of complete stable non-vertical CMC surfaces in $\mm(-1)$ having empty singular set and mean curvature $H\in[0,1)$. In Section~\ref{sec:helicoids} we show existence of a continuum of such surfaces by means of a geometric construction which provides, at the same time, new examples of complete CMC surfaces with empty singular set in $\e$.

Let us motivate and explain our construction in more detail. Take a complete vertical surface $\Sg$ in $\e$ of constant mean curvature $H$. By the characterization result in \cite[Prop.~4.5]{rosales} the surface $\Sg$ is generated by some CC-geodesic $\ga$ of curvature $H$ in $\e$ by means of vertical translations. Moreover, after some computations we get $\ele(\mnh)=4(H^2+\kappa)$ along $\Sg$. Hence, our stability criterion in Corollary~\ref{cor:criteria} (iii) implies that $\Sg$ is strongly stable provided $H^2+\kappa\leq 0$. Now, the idea is to produce a deformation $\Sg'$ of $\Sg$ when $H^2+\kappa<0$, in such a way that $\Sg'$ has empty singular set (though it is possibly non vertical) and still satisfies $\ele(\mnh)\leq 0$. Observe that a natural deformation of $\Sg$ arises when an arbitrary one-parameter family of vertical screw motions of $\e$ is acting on the generating CC-geodesic $\ga$. In an equivalent way, we are led to introduce in \eqref{eq:sgla} the sets $\Sg_{\la,\sg}$ obtained when we leave orthogonally from the vertical axis of $\e$ by CC-geodesics of the same curvature $\la$ and initial velocity determined by a $C^n$ angle function $\sg$ with $n\geq 1$. We must remark that the sets $\Sg_{0,\sg}$ in $\mm(0)$ were previously studied by Cheng and Hwang~\cite{ch} in their classification of properly embedded minimal surfaces of helicoid type. Moreover, when $\cos\sg\neq 0$, the sets $\Sg_{0,\sg}$ coincide with the entire graphical strips introduced by Danielli, Garofalo, Nhieu and Pauls \cite{dgnp} when they solved the Bernstein problem in $\mm(0)$ for $C^2$ graphs with empty singular set.

In Theorem~\ref{th:main1} we employ CC-Jacobi fields to prove some properties of $\Sg_{\la,\sg}$. We analyze when these sets provide immersed surfaces with empty singular set. In such cases, we show that $\Sg_{\la,\sg}$ is complete, orientable, and has constant mean curvature $\la$. By the aforementioned rigidity results none of these surfaces is strongly stable in $\e$, $\kappa\geq 0$, unless $\kappa=0$, $\la=0$ and $\sg$ is constant. Nevertheless, in the hyperbolic model $\mm(-1)$ we are able to give a sufficient condition on the angle function $\sg$ ensuring that $\ele(\mnh)\leq 0$, which ensures by Corollary~\ref{cor:criteria} (iii) that $\Sg_{\la,\sg}$ is strongly stable. This allows us to deduce in Example~\ref{ex:minhyper} that, for any $H\in[0,1)$, there is a complete strongly stable non-vertical surface $\Sg$ in $\mm(-1)$ with $\Sg_0=\emptyset$ and constant mean curvature $H$. Indeed, an embedded right handed helicoid is a minimal example in the previous conditions. We would like to emphasize that the existence of these examples is in clear contrast with the situation in $\mm(0)$, where previous results guarantee that any complete strongly stable surface $\Sg$ with $\Sg_0=\emptyset$ is a vertical plane. It is also interesting to observe that the solution to the Bernstein problem in $\mm(0)$ given in \cite{dgnp} has no direct counterpart in $\mm(-1)$. We illustrate this fact in Example~\ref{ex:bernstein}, since the surface $y=xt$ is a strongly stable entire minimal graph with empty singular set and different from a vertical plane.

The family $\Sg_{\la,\sg}$ contains new examples of complete CMC surfaces with empty singular set in $\e$. We must point out that the classification of such surfaces is far from being established. In $\mathbb{M}(0)$ some partial results were obtained for minimal surfaces, \cite{ch}, \cite{bscv}, and for CMC surfaces of revolution \cite{rr1}. On the other hand, complete CMC vertical surfaces in Sasakian sub-Riemannian $3$-manifolds were described in \cite{rosales}. Following this line, in Theorem~\ref{th:main1} (vi) we establish that any $C^1$ complete CMC surface in $\e$ having empty singular set and containing a vertical line must be congruent to some surface $\Sg_{\la,\sg}$.

In the sub-Riemannian $3$-sphere $\mm(1)$ the study of compact CMC surfaces is particularly interesting. As a consequence of \cite[Thm.~E]{chmy}, we know that a compact CMC surface $\Sg$ with $\Sg_0=\emptyset$ is topologically a torus. As the authors proved in \cite{hr1}, if we further assume the mean curvature $H$ to satisfy $H/\sqrt{1+H^2}\in\rr\setminus\mathbb{Q}$, then $\Sg$ is congruent to a vertical Clifford torus. This result does not hold if $H/\sqrt{1+H^2}\in\mathbb{Q}$, as it is illustrated by the rotationally symmetric examples classified in \cite{hr1}. In Example~\ref{ex:newtori} of the present paper we find CMC tori in the family $\Sg_{\la,\sg}$ which are neither rotationally symmetric with respect to the vertical axis nor congruent to a vertical Clifford torus. This fact might suggest that the family of CMC tori with empty singular set in $\mm(1)$ is very large. However, as we showed in \cite{rosales} none of these tori is stable.

In Section~\ref{sec:nonempty} we discuss stability properties for the class of complete volume-preserving area-stationary surfaces with non-empty singular set in sub-Riemannian $3$-space forms. In the Heisenberg group $\mm(0)$ it is possible to find many examples of area-minimizing surfaces with low analytical regularity, see for instance \cite{pauls-regularity}, \cite{chy}, \cite{mscv} and \cite{r2}. However, by assuming $C^2$ regularity the situation is considerably more rigid. Indeed, any $C^2$ surface $\Sg$ in the class above can be fully described in terms of the singular set $\Sg_0$ and the mean curvature $H=\la$. From the work of Cheng, Hwang, Malchiodi and Yang~\cite{chmy}, the set $\Sg_0$ can contain isolated points and/or $C^1$ curves with non-vanishing tangent vector. We will analyze the two cases separately.

If there is an isolated point in $\Sg_0$ then we have by \cite{hr2} that $\Sg$ is congruent either to a planar surface $\pla$ as in \eqref{eq:planes}, or to a spherical surface $\sla(p)$ as in \eqref{eq:spheres}. The spheres $\sla(p)$ are generalizations of the Pansu spheres in the Heisenberg group $\mm(0)$. In \cite{hr2} we proved that, though $\sla(p)$ is not strongly stable, it is a second order minima of the area under volume-preserving admissible variations. In the present paper we complete the stability question in this setting by showing in Theorem~\ref{th:stabilitypla} that the planes $\pla$ are \emph{strictly stable}. This means that $(A+2\la V)''(0)\geq 0$ for any admissible variation of $\pla$, and equality holds if and only if the associated velocity vector field is always tangent to $\pla$. We remark that employing admissible variations is necessary to apply the second variation formula \eqref{eq:intro1} since we do not require the deformations to fix the singular point $p$. In fact, due to the integrability condition in Lemma~\ref{lem:comppla} (i), the family of admissible variations of $\pla$ moving $p$ is very large, see \cite[App.~B]{hr2}. It is worth mentioning that the stability of the minimal plane $\mathcal{P}_0(p)$ in $\mm(0)$ is very well known; as a matter of fact $\mathcal{P}_0(p)$ is area-minimizing by a calibration argument, see \cite{rr2} and \cite{bscv}. From a similar reasoning we may conclude the stability of the minimal planes $\mathcal{P}_0(p)$ in the hyperbolic model $\mm(-1)$. We stress that Theorem~\ref{th:stabilitypla} is valid for all planes $\pla$ in sub-Riemannian $3$-space forms of arbitrary topology.

The proof of Theorem~\ref{th:stabilitypla} has two steps. In the first one, we employ the same arguments as for the spheres $\sla(p)$ in \cite{hr2} to infer that the divergence term in \eqref{eq:intro1} vanishes. In the second step we see that $\mathcal{Q}(w,w)\geq 0$ for any $w\in C_0^1(\pla)$, with equality if and only if $w=0$. Observe that, since the planes $\pla$ have no vertical points, Corollary~\ref{cor:criteria} (i) gives inequality $\mathcal{Q}(w,w)\geq 0$ provided the function $w$ vanishes off of the pole $p$. The proof that $\mathcal{Q}(w,w)\geq 0$ for any $w\in C_0^1(\pla)$ is more technical and relies on Lemma~\ref{lem:newibp}, where we use the analytical behaviour of $\pla$ around $p$ to deduce an integration by parts formula involving the Jacobi function $\psi=\escpr{N,T}$. From Lemma~\ref{lem:newibp} we can obtain Theorem~\ref{th:stabilitypla} just by reproducing the proof of Theorem~\ref{th:criterion}.

On the other hand, if the singular set $\Sg_0$ contains a curve, then there is an ambient CC-geodesic $\Ga$ such that $\Sg$ is congruent to the surface $\mathcal{C}_{\la}(\Ga)$ defined in \cite[Thm.~4.13]{hr2}, see also the refe\-rences therein. Roughly speaking, the surfaces $\cmula(\Ga)$ are produced by matching together some ``fundamental pieces" in a suitable way. These pieces are surfaces $\Sg_\la(\Ga)$ as in \eqref{eq:sglaGa}, and they are obtained when one leaves from $\Ga$ by a family of orthogonal CC-geodesic segments of a given curvature and length.

In general, we cannot expect the surfaces $\cmula(\Ga)$ to be strongly stable. Consider the simplest case of a CC-geodesic $\Ga$ of curvature $\mu$ in the Heisenberg group $\mm(0)$. If $\mu=0$ then $\mathcal{C}_0(\Ga)$ is congruent to the hyperbolic paraboloid $t=xy$, and it is strongly stable (indeed area minimizing), see \cite{rr2}. In the case $\mu\neq 0$, Ritor\'e and the authors~\cite{hrr} found a variation of $\mathcal{C}_{0}(\Ga)$ moving the two singular curves of $\mathcal{C}_0(\Ga)$ while strictly decreasing the area. This leads us to study the stability of $\cmula(\Ga)$ under variations supported off of the singular set. Surprisingly, in Theorem~\ref{th:cmulastability} we are able to deduce, as an immediate consequence of Corollary~\ref{cor:criteria} (iii), that the regular set of any surface $\cmula(\Ga)$ within a $3$-dimensional space form is strongly stable. This implies in particular that, in order to show instability of $\cmula(\Ga)$, one needs to use suitable variations moving the singular set.

The paper is organized into five sections. Section~\ref{sec:preliminaries} contains background material about sub-Riemannian $3$-manifolds. In Section~\ref{sec:statstable} we introduce the variational setting and prove our stability criteria for CMC surfaces in Sasakian sub-Riemannian $3$-manifolds. In Section~\ref{sec:helicoids} we construct and classify complete CMC surfaces in $\e$ containing a vertical line, focusing on strongly stable examples in $\mm(-1)$. Finally, in Section~\ref{sec:nonempty} we establish our stability results for the surfaces $\mathcal{P}_\la(p)$ and $\mathcal{C}_{\la}(\Ga)$.

\section{Preliminaries}
\label{sec:preliminaries}
\setcounter{equation}{0}

In this section we introduce the notation and gather some known results that will be used throughout the paper.

\subsection{Sasakian sub-Riemannian $3$-manifolds}
\label{subsec:ssR3m}
\noindent

A \emph{contact sub-Riemannian manifold} is a connected manifold $M$ with $\ptl M=\emptyset$ together with a Riemannian metric $g_h$ defined on an oriented contact distribution $\h$, which we refer to as the \emph{horizontal distribution}. A vector field $U$ is \emph{horizontal} if it coincides with its projection $U_h$ onto $\h$.

The \emph{normalized form} is the contact $1$-form $\eta$ on $M$ such that $\text{Ker}(\eta)=\h$ and the restriction of the $2$-form $d\eta$ to $\h$ equals the area form on $\h$. We will consider the orientation of $M$ induced by $\eta\wedge d\eta$. The \emph{Reeb vector field} is the vector field $T$ transversal to $\h$ defined by $\eta(T)=1$ and $d\eta(T,U)=0$, for any $U$. If $U$ is always proportional to $T$ then we say that $U$ is \emph{vertical}.

We denote by $J$ the orientation-preserving $90$ degree rotation in $(\h,g_h)$. This is a contact structure on $\h$ since $J^2=-\text{Id}$. We extend $J$ to the tangent bundle of $M$ by setting $J(T):=0$.

The \emph{canonical extension} of $g_h$ is the Riemannian metric $g=\escpr{\cdot\,,\cdot}$ on $M$ extending $g_h$, and such that $T$ is a unit vector field orthogonal to $\h$. The \emph{length} of a vector field $U$ is $|U|:=\escpr{U,U}^{1/2}$.
We say that $M$ is \emph{complete} if $(M,g)$ is a complete Riemannian manifold.

By an \emph{isometry} of $M$ we mean a $C^\infty$ diffeomorphism $\phi:M\to M$ whose differential at any $p\in M$ is an orientation-preserving linear isometry from $\h_p$ to $\h_{\phi(p)}$. We say that $M$ is \emph{homogeneous} if the group $\text{Iso}(M)$ of isometries of $M$ acts transitively on $M$. In such a case $M$ must be complete. Two subsets $S_1$ and $S_2$ of $M$ are \emph{congruent} if there is $\phi\in\text{Iso}(M)$ such that $\phi(S_1)=S_2$.

By a \emph{Sasakian sub-Riemannian $3$-manifold} we mean a contact sub-Riemannian $3$-manifold $M$ such that $g_h$ is a \emph{Sasakian metric}, i.e., any diffeomorphism of the one-parameter group of $T$ is an isometry. This implies that $(M,g)$ is a K-contact Riemannian manifold \cite[Sect.~6.2]{blair}. It follows from \cite[p.~67, Cor.~6.5, Thm.~6.3]{blair} that the Levi-Civit\`a connection $D$ associated to $g$ satisfies
\begin{align}
\nonumber
D_UT&=J(U),
\\
\label{eq:dujv}
D_U\left(J(V)\right)&=J(D_UV)+\escpr{V,T}\,U-\escpr{U,V}\,T.
\end{align}
In particular, the integral curves of $T$ are geodesics in $(M,g)$ parameterized by arc-length. We refer to these curves as \emph{vertical lines}.

The \emph{Webster scalar curvature} $K$ of a contact sub-Riemannian $3$-manifold $M$ is the sectional curvature of $\h$ with respect to the Tanaka connection \cite[Sect.~10.4]{blair}. If $M$ is Sasakian then we have $K=(1/4)\,(K_h+3)$, where $K_h$ denotes the sectional curvature of $\h$ in $(M,g)$.

\subsection{Three-dimensional space forms.}
\label{subsec:model}
\noindent

For $\kappa=-1,0,1$, we denote by $\m$ the complete, simply connected, Riemannian surface of constant sectional curvature $4\kappa$ described as follows. If $\kappa=1$ then $\m$ is the unit sphere $\mathbb{S}^2\sub\rr^3$ with its standard Riemannian metric scaled by $1/4$. If $\kappa=-1,0$ then $\m:=\{p\in\rr^2\,;\,|p|<1/|\kappa|\}$ endowed with the Riemannian metric $\rho^2\,(dx^2+dy^2)$, where $\rho(x,y):=(1+\kappa(x^2+y^2))^{-1}$. Note that $\mathbb{N}(-1)$ is the Poincar\'e model of the hyperbolic plane and $\mathbb{N}(0)$ is the Euclidean plane.

For $\kappa=-1,0$ we denote $\e:=\m\times\rr$. Let $(x,y,t)$ be the Euclidean coordinates in $\rr^3$. We define in $\e$ the planar distribution $\h:=\text{Ker}(\eta)$, where $\eta:=\rho\,(x\,dy-y\,dx)+dt$ and $\rho(x,y,t):=\rho(x,y)$. A basis $\{X,Y,T\}$ of vector fields on $\e$ such that $X$, $Y$ are sections of $\h$ and $T$ is the Reeb vector field associated to the contact $1$-form $\eta$ is given by
\begin{align*}
X&:=\frac{1}{\rho}\left(\cos(2\kappa t)\,\frac{\ptl}{\ptl x}-\sin(2\kappa t)\,\frac{\ptl}{\ptl y}\right)
+\left(y\,\cos(2\kappa t)+x\,\sin(2\kappa t)\right)\frac{\ptl}{\ptl t},
\\
Y&:=\frac{1}{\rho}\left(\sin(2\kappa t)\,\frac{\ptl}{\ptl x}+\cos(2\kappa t)\,\frac{\ptl}{\ptl y}\right)
+\left(y\,\sin(2\kappa t)-x\,\cos(2\kappa t)\right)\frac{\ptl}{\ptl t},
\\
T&:=\frac{\ptl}{\ptl t}.
\end{align*}
For $\kappa=1$ we consider $\mathbb{M}(\kappa):=\mathbb{S}^3$ with the planar distribution $\h:=\text{Ker}(\eta)$, where $\eta:=x_1\,dy_1-y_1\,dx_1+x_2\,dy_2-y_2\,dx_2$. Here $(x_1,y_1,x_2,y_2)$ are  the Euclidean coordinates in $\rr^4$.  A basis $\{X,Y,T\}$ of vector fields in the same conditions as above is defined by
\begin{align*}
X&:=-x_{2}\,\frac{\ptl}{\ptl x_{1}}+y_{2}\,\frac{\ptl}{\ptl
y_{1}}+x_{1}\,\frac{\ptl}{\ptl x_{2}}-y_{1}\,\frac{\ptl}{\ptl y_{2}},
\\
Y&:=-y_{2}\,\frac{\ptl}{\ptl x_{1}}-x_{2}\,\frac{\ptl}{\ptl
y_{1}}+y_{1}\,\frac{\ptl}{\ptl x_{2}}+x_{1}\,\frac{\ptl}{\ptl y_{2}},
\\
T&:=-y_{1}\,\frac{\ptl}{\ptl x_{1}}+x_{1}\,\frac{\ptl}{\ptl y_{1}}
-y_{2}\,\frac{\ptl}{\ptl x_{2}}+x_{2}\,\frac{\ptl}{\ptl y_{2}}.
\end{align*}
Some easy computations show the following bracket relations
\begin{equation}
\label{eq:lb}
[X,Y]=-2\,T,\quad [X,T]=(2\kappa)\,Y,\quad [Y,T]=-(2\kappa)\,X,\quad\kappa=-1,0,1.
\end{equation}
We consider the orientation in $\h$ (resp. $\e$) for which $\{X_p,Y_p\}$ (resp. $\{X_p,Y_p,T_p\}$) is a positive basis of $\h_p$ (resp. $T_p\e$) at any $p\in\e$. We take the Riemannian metric $g_h$ on $\h$ for which $\{X_p,Y_p\}$ is an orthonormal basis at any $p\in\e$. Hence, the associated orientation-preserving 90 degree rotation $J$ satisfies $J(X_p)=Y_p$ and $J(Y_p)=-X_p$.

In $\e$ there is a product $*$ such that $(\e,*)$ is the Heisenberg group when $\kappa=0$, the group of unit quaternions when $\kappa=1$, and the universal covering of the special linear group $\text{SL}(2,\rr)$ when $\kappa=-1$. The identity element for $*$ is the point $o:=(0,0,0)$ when $\kappa=-1,0$, or $o:=(1,0,0,0)$ when $\kappa=1$. The associated left translations (resp. right translations) are isometries of $\e$ when $\kappa=-1,0$ (resp. $\kappa=1$). Hence the isometry group of $\e$ acts transitively on $\e$. As a consequence $\e$ is homogeneous and, in particular, complete.

For any $s\in\rr$, we define the \emph{vertical translation} $\phi_s:\e\to\e$ as the map $\phi_s(p):=p*s\,T_p=p+s\,T_p$ if $\kappa=-1,0$, or $\phi_s(p):=e^{is}*p$ if $\kappa=1$. It is easy to see that $\{\phi_s\}_{s\in\rr}$ is the one-parameter group of diffeomorphisms associated to $T$. Since any $\phi_s$ is an isometry of $\e$ we conclude that $\e$ is a Sasakian sub-Riemannian $3$-manifold. The vertical lines in $\e$ parameterize straight lines if $\kappa=-1,0$ or great circles in $\stres$ if $\kappa=1$. The \emph{vertical axis} of $\e$ is the vertical line passing through $o$. By means of a left or right translation any vertical line is congruent to the vertical axis. We can check that a \emph{vertical rotation}, i.e., a Euclidean rotation about the vertical axis is an isometry of $\e$. It follows that any \emph{vertical screw motion}, defined as the composition of a vertical rotation and a vertical translation, is also an isometry of $\e$. We say that a set $S\subset\e$ is \emph{rotationally symmetric} if it is invariant under vertical rotations.

The space $\e$ has constant Webster scalar curvature $\kappa$, see \cite[Sect.~7.4]{blair} and \cite[Sect.~2.2]{rosales}. Indeed, a result of Tanno~\cite{tanno} establishes that $\e$ is, up to isometries, the unique complete, simply connected, Sasakian sub-Riemannian $3$-manifold of Webster scalar curvature $\kappa$. As in Riemannian geometry one can construct Sasakian sub-Riemannian $3$-manifolds of constant curvature and non-trivial topology, see \cite[Prop.~2.1, Ex.~2.2]{hr2}. In the sequel, by a \emph{$3$-dimensional space form} we mean a complete Sasakian sub-Riemannian $3$-manifold of constant Webster scalar curvature.

\newpage

\subsection{Carnot-Carath\'eodory geodesics and Jacobi fields}
\label{subsec:geodesics}
\noindent

Let $M$ be a Sasakian sub-Riemannian $3$-manifold. A \emph{horizontal curve} in $M$ is a $C^1$ curve $\ga$ whose velocity vector $\dot{\ga}$ is horizontal. The \emph{length} of $\ga$ over an interval $[a,b]$ is $\int_a^b|\dot{\ga}(s)|\,ds$. Following the approach in \cite[Sect.~3]{rr2} and \cite[Sect.~3]{rosales}, we say that a $C^2$ horizontal curve $\ga$ parameterized by arc-length is a \emph{Carnot-Carath\'eodory geodesic}, or simply a \emph{$CC$-geodesic}, if it is a critical point of length under $C^2$ variations by horizontal curves. As in \cite[Prop.~3.1]{rr2} this is equivalent to the existence of a constant $\la\in\rr$, called the \emph{curvature} of $\ga$, such that the second order ODE\begin{equation}
\label{eq:geoeq}
\dot{\ga}'+2\la\,J(\dot{\ga})=0
\end{equation}
is satisfied. Here the prime $'$ stands for the covariant derivative along $\ga$ in $(M,g)$. It follows that any CC-geodesic is a $C^\infty$ curve. If $p\in M$ and $v\in\h_p$ with $|v|=1$, then the unique maximal solution $\ga$ to \eqref{eq:geoeq} with $\ga(0)=p$ and $\dot{\ga}(0)=v$ is a CC-geodesic of curvature $\la$ since $\escpr{\dot{\ga},T}$ and $|\dot{\ga}|^2$ are constant functions along $\ga$. If $M$ is complete then any maximal CC-geodesic is defined on $\rr$, see \cite[Thm.~1.2]{falbel4}.

In \cite[Lem.~3.5]{rr2} and \cite[Lem.~3.3]{rosales} the \emph{CC-Jacobi fields} were introduced as infinitesimal vector fields associated to one-parameter families of CC-geodesics of the same curvature. In the next result, which follows from \cite[Lem.~3.3, Lem.~ 3.4]{rosales}, we gather some properties of these vector fields.

\begin{lemma}
\label{lem:ccjacobi}
Let $M$ be a complete Sasakian sub-Riemannian $3$-manifold. Consider a $C^1$ curve $\alpha:I\to M$ defined on some open interval $I\subseteq\rr$, and a $C^1$ unit horizontal vector field $U(\eps)$ along $\alpha$. For a fixed $\lambda\in\rr$, we define the map $F:I\times\rr\to M$ by $F(\eps,s):=\ga_{\eps}(s)$, where $\ga_{\eps}(s)$ is the CC-geodesic of curvature $\la$ with $\ga_{\eps}(0)=\alpha(\eps)$ and $\dot{\ga}_{\eps}(0)=U(\eps)$. Then, the vector field $V_{\eps}(s):=(\ptl F/\ptl\eps)(\eps,s)$ and the function $v_\eps(s):=\escpr{V_\eps(s),T}$ satisfy these properties:
\begin{itemize}
\item[(i)] $V_\eps$ is $C^\infty$ along $\ga_\eps$ with $[\dot{\ga}_\eps,V_\eps]=0$,
\item[(ii)] the expression of $V_\eps$ with respect to the orthonormal basis
$\{\dot{\ga}_\eps,J(\dot{\ga}_\eps),T\}$ is
\[
V_\eps=\left\{\la\left(\escpr{\dot{\alpha}(\eps),T}-v_\eps\right)+\escpr{\dot{\alpha}(\eps),U(\eps)}\right\}\dot{\ga}_\eps+(v'_\eps/2)\,J(\dot{\ga}_\eps)+v_\eps\,T,
\]
where the prime $'$ denotes the derivative with respect to $s$,
\item[(iii)] if we denote $\tau:=4\,(\la^2+K)$, then $v'''_\eps+\tau\,v'_\eps=0$ along $\ga_\eps$. Hence, if $M$ has constant Webster scalar curvature $K$, then we have:
\begin{itemize}
\item[(a)] for $\tau<0$,
\[
v_\eps(s)=\frac{1}{\sqrt{-\tau}}\left(a_\eps\,\sinh(\sqrt{-\tau}\,s)+b_\eps\,\cosh(\sqrt{-\tau}\,s)\right)+c_\eps,
\]
where $a_\eps=v_\eps'(0)$, $b_\eps=(1/\sqrt{-\tau})\,v_\eps''(0)$ and $c_\eps=v_\eps(0)+(1/\tau)\,v_\eps''(0)$,
\vspace{0,1cm}
\item[(b)] for $\tau=0$,
\[
v_\eps(s)=a_\eps\,s^2+b_\eps\,s+c_\eps,
\]
where $a_\eps=(1/2)\,v_\eps''(0)$, $b_\eps=v_\eps'(0)$ and $c_\eps=v_\eps(0)$,
\vspace{0,1cm}
\item[(c)] for $\tau>0$,
\[
v_\eps(s)=\frac{1}{\sqrt{\tau}}\left(a_\eps\,\sin(\sqrt{\tau}\,s)-b_\eps\,\cos(\sqrt{\tau}\,s)\right)+c_\eps,
\]
where $a_\eps=v_\eps'(0)$, $b_\eps=(1/\sqrt{\tau})\,v_\eps''(0)$ and $c_\eps=v_\eps(0)+(1/\tau)\,v_\eps''(0)$.
\end{itemize}
\end{itemize}
\end{lemma}

\subsection{Horizontal geometry of surfaces}
\label{subsec:horizontal}
\noindent

Let $\Sg$ be a $C^1$ surface immersed in a Sasakian sub-Riemannian $3$-manifold $M$. Unless explicitly stated we will assume that $\ptl\Sg=\emptyset$. The \emph{singular set} $\Sg_0$ of $\Sg$ consists of those points $p\in\Sg$ for which the tangent plane $T_p\Sg$ equals the horizontal plane $\h_{p}$. Since $\h$ is a completely nonintegrable distribution, it follows by Frobenius theorem that $\Sg_0$ is closed and has empty interior in $\Sg$. Hence the \emph{regular set} $\Sg-\Sg_0$ of $\Sg$ is open and dense in $\Sg$. By using the arguments in \cite[Lem.~1]{d2}, see also \cite[Thm.~1.2]{balogh} and \cite[App.~A]{hp2}, we deduce that, for a $C^2$ surface $\Sg$, the Hausdorff dimension of $\Sg_{0}$ in $(M,g)$ is less than or equal to $1$.  In particular, the Riemannian area of $\Sg_{0}$ vanishes.

Suppose that $\Sg$ is oriented, and denote by $N$ the unit normal along $\Sg$ in $(M,g)$ which is compatible with the orientations of $\Sg$ and $M$. We define the (sub-Riemannian) \emph{area} of $\Sg$ by
\begin{equation}
\label{eq:area}
A(\Sg):=\int_{\Sg}|N_{h}|\,da,
\end{equation}
where $N_h=N-\escpr{N,T}\,T$ and $da$ is the area element in $(M,g)$. In case $\Sg$ bounds a set $\Om\subset M$, then $A(\Sg)$ coincides with the \emph{sub-Riemannian perimeter \'a la De Giorgi} of $\Om$, which can be introduced as in \cite{cdg1}. Note that $\Sg_{0}=\{p\in\Sg\,;N_h(p)=0\}$. In the regular set $\Sg-\Sg_0$, we can define the \emph{horizontal Gauss map} $\nu_h$ and the \emph{characteristic vector field} $Z$, by
\begin{equation}
\label{eq:nuh}
\nu_h:=\frac{N_h}{|N_h|}, \qquad Z:=J(\nuh).
\end{equation}
As $Z$ is horizontal and orthogonal to $\nu_h$ then $Z$ is tangent to $\Sg$. Hence $Z_{p}$ generates $T_{p}\Sg\cap\h_{p}$ for any $p\in\Sg-\Sg_0$. The integral curves of $Z$ in $\Sg-\Sg_0$ will be called $(\emph{oriented}\,)$ \emph{characteristic curves} of $\Sg$.  They are
both tangent to $\Sg$ and to $\h$.  If we define
\begin{equation}
\label{eq:ese}
S:=\escpr{N,T}\,\nu_h-|N_h|\,T,
\end{equation}
then $\{Z_{p},S_{p}\}$ is an orthonormal basis of $T_p\Sg$ whenever
$p\in\Sg-\Sg_0$.  Moreover, for any $p\in\Sg-\Sg_{0}$ we have the
orthonormal basis of $T_{p}M$ given by $\{Z_{p},(\nuh)_{p},T_{p}
\}$.  From here we deduce that
\begin{equation}
\label{eq:relations}
|N_{h}|^2+\escpr{N,T}^2=1, \qquad (\nu_{h})^\top=\escpr{N,T}\,S,
\qquad T^\top=-|N_{h}|\,S,
\end{equation}
on $\Sg-\Sg_0$, where $U^\top$ stands for the projection of a vector field $U$ onto the tangent plane to $\Sg$.

A \emph{vertical point} is a point $p\in\Sg$ such that $T_p\in T_p\Sg$. This is equivalent to that $\escpr{N_p,T_p}=0$. We say that $\Sg$ is a \emph{vertical surface} if any $p\in\Sg$ is a vertical point, i.e., $\escpr{N,T}=0$ along $\Sg$.

If $\Sg$ is an oriented $C^2$ surface immersed in $M$ then, for any
$p\in\Sg-\Sg_0$ and $v\in T_pM$, we have these equalities, see \cite[Lem.~3.5]{hrr} and \cite[Lem.~4.2]{rosales}
\begin{align}
\label{eq:vmnh}
v\,(|N_h|)&=\escpr{D_{v}N, \nu_{h}}+\escpr{N,T}\,
\escpr{v,Z},
\\
\label{eq:vnt}
v(\escpr{N,T})&=\escpr{D_{v}N,T}+\escpr{N,J(v)},
\\
\label{eq:dvnuh}
D_{v}\nu_h&=|N_h|^{-1}\, \big(\escpr{D_vN,Z}-\escpr{N,T}\,
\escpr{v,\nuh}\big)\,Z+\escpr{v,Z}\,T.
\end{align}

We denote by $B$ the \emph{shape operator} of $\Sg$ in $(M,g)$.  It is given by $B(U):=-D_{U}N$,  for any vector $U$ tangent to $\Sg$. Finally, we will say that $\Sg$ is \emph{complete} if it is complete with respect to the Riemannian metric induced by $g$.

\section{Strongly stable surfaces}
\label{sec:statstable}

In this section we establish sufficient conditions ensuring that a surface inside a Sasakian sub-Riemannian $3$-manifold is strongly stable (the precise definition is given in Section~\ref{subsec:stable} below). We begin by recalling some known facts about critical points and second order minima of the area with or without a volume constraint.

\subsection{Area-stationary surfaces}
\label{subsec:stationary}
\noindent

Let $M$ be a Sasakian sub-Riemannian $3$-manifold and $\varphi_0:\Sg\to M$ an oriented $C^1$ surface immersed in $M$. By a \emph{variation} of $\Sg$ we mean a $C^1$ map $\varphi:I\times\Sg\to M$, where $I\subeq\rr$ is an open interval containing $0$, and $\varphi$ satisfies:
\begin{itemize}
\item[(i)] $\var(0,p)=\varphi_0(p)$ for any $p\in\Sg$,
\item[(ii)] the map $\varphi_s:\Sg\to M$ given by $\varphi_s(p):=\varphi(s,p)$ is an immersion for any $s\in I$,
\item[(iii)] there is a compact set $C\subseteq\Sg$ such that $\varphi_{s}(p)=\varphi_0(p)$ for any $s\in I$ and any $p\in\Sg-C$.
\end{itemize}
Observe that we do not assume neither a specific form for $\varphi$, nor that the \emph{velocity vector field} $U_p:=(\ptl\varphi/\ptl s)(0,p)$ is always normal to $\Sg$ or proportional to the horizontal Gauss map $\nuh$ in \eqref{eq:nuh}.

The \emph{area functional} associated to $\varphi$ is given by $A(s):=A(\Sg_{s})$, see \eqref{eq:area}.  We consider the \emph{volume functional} $V(s)$ defined in \cite[Sect.~2]{bdce} as the signed volume in $(M,g)$ between $\Sg$ and $\Sg_s$. More precisely
\begin{equation}
\label{eq:volume}
V(s):=\int_{[0,s]\times C}\varphi^*(dv),
\end{equation}
where $dv$ denotes the volume element in $(M,g)$. A variation $\varphi$ is \emph{volume preserving} if $V(s)$ is constant for any $s$ small enough. We say that $\Sg$ is \emph{area-stationary} if $A'(0)=0$ for any variation of $\Sg$.
We say that $\Sg$ is \emph{volume-preserving area-stationary} or \emph{area-stationary under a volume constraint} if $A'(0)=0$ for any volume-preserving variation of $\Sg$.

We denote by $N$ the unit normal vector along $\Sg$ in $(M,g)$ which is compatible with the orientations of $\Sg$ and $M$. If $\Sg$ is $C^2$ on $\Sg-\Sg_0$, then the (sub-Riemannian) \emph{mean curvature} of $\Sg$ is the function $H$ on $\Sg-\Sg_0$ defined as in \cite{rr1} and \cite{rosales} by
\[
\label{eq:mc}
-2H:=\divv_\Sg\nuh=\escpr{D_{Z}\nuh,Z}+\escpr{D_{S}\nuh, S},
\]
where $\{Z,S\}$ is the orthonormal basis of the tangent plane to $\Sg-\Sg_0$ defined in \eqref{eq:nuh} and \eqref{eq:ese}. From \eqref{eq:dvnuh} it follows that $D_S\nuh$ is proportional to $Z$, and so $\escpr{D_{S}\nuh,S}=0$. Since $\escpr{Z,\nuh}=0$ we deduce the identity
\begin{equation}
\label{eq:mc2}
2H=\escpr{D_{Z}Z,\nuh} \ \text{on} \ \Sg-\Sg_0.
\end{equation}
Suppose now that $\Sg$ is $C^1$ and the vector field $Z$ is $C^1$ along the characteristic curves. In such a case, we define the mean curvature of $\Sg$ by means of equality \eqref{eq:mc2}. This coincides, up to a factor, with the definition in \cite[Eq.~(5.1)]{galli-ritore2}. We say that $\Sg$ has \emph{constant mean curvature} (CMC) if $H$ is constant on $\Sg-\Sg_0$. When $H=0$ on $\Sg-\Sg_0$ we say that $\Sg$ is a \emph{minimal surface}. The next result characterizes $C^1$ volume-preserving area-stationary surfaces with empty singular set.

\begin{proposition}
\label{prop:vpstationary}
Let $\Sg$ be an oriented $C^1$ surface immersed in a Sasakian sub-Riemannian $3$-manifold $M$. Suppose that $\Sg_0=\emptyset$ and the vector field $Z$ is $C^1$ along the characteristic curves. Then, $\Sg$ is volume-preserving area-stationary $($resp. area-stationary$)$ if and only if $H$ is constant $($resp. $H=0$$)$ on $\Sg$. In such a case, given any point $p\in\Sg$, there is a unique characteristic curve $\ga$ through $p$, and $\ga$ is a CC-geodesic in $M$ of curvature $H$.
\end{proposition}

This proposition is well known for $C^2$ surfaces, see \cite[Sect.~4.1]{hr2} and the references therein. In the $C^1$ case the first part of the statement is deduced from \cite[Sect.~3, Cor.~5.2]{galli-ritore2}, see also \cite[Prop.~6.3, Re.~6.4]{galli}. The second part was obtained in \cite{chy2} for the Heisenberg group $\mm(0)$. The uniqueness of the characteristic curves is found in the proof of \cite[Thm.~4.1]{galli-ritore2}. Finally, if $\ga$ is a characteristic curve, then the regularity result in \cite[Prop.~5.4]{galli-ritore2} implies that $\ga$ is a $C^\infty$ curve. Now, we can proceed as in the case where $\Sg$ is $C^2$ to conclude that $\ga$ is a CC-geodesic of curvature $H$.

The interested reader is referred to a recent work of Galli~\cite{galli3} for a generalization in contact sub-Riemannian $3$-manifolds for domains with prescribed mean curvature and Lipschitz boundary which is locally a regular intrinsic graph.

\subsection{Stability and second variation formula}
\label{subsec:stable}
\noindent

In this section we will assume more analytical regularity than in the previous one. For our purposes in the paper, it will suffice to consider $C^2$ surfaces $\Sg$ such that $\Sg-\Sg_0$ is $C^3$, and $C^2$ variations $\varphi$ of class $C^3$ off of $\Sg-\Sg_0$.

Let $\Sg$ be an oriented surface immersed in a Sasakian sub-Riemannian $3$-manifold $M$. Suppose that $\Sg$ is area-stationary with or without a volume constraint, and let $H$ be the constant mean curvature along $\Sg-\Sg_0$. From the first variational formulas for area and volume, see for instance \cite[Sect.~4.1]{hr2}, it follows that $(A+2HV)'(0)=0$ for any variation of $\Sg$. We say that $\Sg$ is
\begin{itemize}
\item[(i)]  \emph{stable} if $(A+2HV)''(0)\geq 0$ for any volume-preserving variation of $\Sg$.
\item[(ii)] \emph{strongly stable} if $(A+2HV)''(0)\geq 0$ for any variation of $\Sg$.
\item[(iii)] \emph{strictly stable} if $(A+2HV)''(0)\geq 0$ for any variation of $\Sg$, and equality holds if and only if the associated velocity vector field $U$ is tangent to $\Sg$.
\end{itemize}

For an area-stationary surface to be strongly stable agrees with the usual notion of stability for minimal surfaces in Riemannian geometry. Note that a strictly stable surface is in particular a strict minimum of the area under volume-preserving variations with non-tangent velocity. Obviously strict stability implies strong stability and strong stability implies stability.

In order to analyze the stability conditions we need to compute $(A+2HV)''(0)$. This involves a technical issue since differentiating two times under the integral sign in \eqref{eq:area} may be not possible. This problem can be solved by applying Leibniz's rule when the variation $\varphi$ is \emph{admissible} in the sense of \cite[Def.~5.1]{hr2}. In such a case we deduce from \cite[Thm.~7.1]{hr2} that
\begin{equation}
\label{eq:2ndvar}
(A+2HV)''(0)=\mathcal{Q}(w,w)+\int_\Sg\divv_\Sg G\,da,
\end{equation}
provided all the terms are locally integrable with respect to $da$. In the previous formula, $w$ is the normal component of the velocity vector field $U$ associated to the variation, $\divv_\Sg G$ is the divergence relative to $\Sg$ of a certain tangent $C^1$ vector field $G$ along $\Sg-\Sg_0$ (an explicit expression is found in \cite[Eq.~(7.1)]{hr2}), and $\mathcal{Q}$ is the (sub-Riemannian) \emph{index form} of $\Sg$ given by
\begin{equation}
\label{eq:indexform}
\mathcal{Q}(u,v):=\int_{\Sg}|N_{h}|^{-1}\left\{Z(u)\,Z(v)-
\big(|B(Z)+S|^2+4\,(K-1)\,|N_{h}|^2\big)\,u\,v\right\}da,
\end{equation}
where $K$ is the Webster scalar curvature of $M$ and $B$ is the shape operator of $\Sg$ in $(M,g)$. Note that $\mathcal{Q}(u,v)$ is well defined if, for instance, $u\in C^1_0(\Sg-\Sg_0)$ and $v\in C^1(\Sg-\Sg_0)$.

As was pointed out in \cite[Ex.~8.2]{hr2}, if the variation $\varphi:I\times\Sg\to M$ is compactly supported on $\Sg-\Sg_0$, then there is an interval $I'\sub\sub I$ such that the restriction of $\varphi$ to $I'\times\Sg$ is admissible. Moreover, we can apply the Riemannian divergence theorem to get $\int_{\Sg}\divv_\Sg G\,da=0$. From \eqref{eq:2ndvar} we deduce that
\begin{equation}
\label{eq:empty}
(A+2HV)''(0)=\mathcal{Q}(w,w), \quad\text{provided $\varphi$ is supported on $\Sg-\Sg_0$.}
\end{equation}

\begin{remark}
\label{re:trivial}
As an immediate consequence of \eqref{eq:empty} and \eqref{eq:indexform}, if the function $q:=|B(Z)+S|^2+4\,(K-1)\,|N_{h}|^2$ satisfies $q\leq 0$ (resp. $q<0$), then $\Sg-\Sg_0$ is strongly stable (resp. strictly stable).
\end{remark}

\begin{example}[Vertical surfaces]
\label{ex:vertical}
Let $\Sg$ be a complete orientable CMC vertical surface immersed in $M$ (see \cite[Sect.~4.2]{rosales} for a characterization result). Then, it is easy to check that $|B(Z)+S|^2+4\,(K-1)\,\mnh^2=4\,(H^2+K)$, so that $\Sg$ is strongly stable (resp. strictly stable) provided $H^2+K\leq 0$ (resp. $H^2+K<0$). On the other hand, if $M=\e$ and $H^2+\kappa>0$, then $\Sg$ is unstable by \cite[Thm.~6.7]{rosales}.
\end{example}

If $u\in C^1_0(\Sg-\Sg_0)$ and $v\in C^2(\Sg-\Sg_0)$ then we can use integration by parts formulas, so that the index form can be expressed in the following way, see \cite[Prop.~3.14]{hrr} and \cite[Prop.~5.8]{rosales}
\begin{equation}
\label{eq:ibp}
\mathcal{Q}(u,v)=-\int_{\Sg} u\,\ele(v)\,da,
\end{equation}
where $\ele$ is the (sub-Riemannian) \emph{Jacobi operator} of $\Sg$, defined by
\begin{align}
\label{eq:lu}
\mathcal{L}(\psi):=|N_{h}|^{-1}\big\{Z(Z(\psi))&+2\,|N_{h}|^{-1}\,\escpr{N,T}\,\escpr{B(Z),S}\,Z(\psi)
\\
\nonumber
&+(|B(Z)+S|^2+4\,(K-1)\,|N_{h}|^2)\,\psi\big\}.
\end{align}
By a (sub-Riemannian) \emph{Jacobi function} on $\Sg$ we mean a function $\psi\in C^2(\Sg-\Sg_0)$ with $\ele(\psi)=0$.

In Riemannian geometry it is well known that the normal component along a CMC surface of an ambient Killing field is a Riemannian Jacobi function, see \cite[Prop.~(2.12)]{bdce}. Here we prove in our sub-Riemannian context a similar result for the Reeb vector field.

\begin{lemma}
\label{lem:jacobi}
Let $\Sg$ be an oriented $C^2$ surface immersed in a Sasakian sub-Riemannian $3$-manifold $M$. Suppose that $\Sg-\Sg_0$ is a $C^3$ surface of constant mean curvature $H$. Then, the function $\psi:=\escpr{N,T}$ is a Jacobi function on $\Sg$.
\end{lemma}

\begin{proof}
We compute $\ele(\escpr{N,T})$ from \eqref{eq:lu}. From \eqref{eq:vnt} and the third equality in \eqref{eq:relations}, we obtain
\begin{equation}
\label{eq:znt}
Z(\escpr{N,T})=\mnh\,\big(\escpr{B(Z),S}-1\big).
\end{equation}
By differentiating into the previous formula, we get
\[
Z(Z(\escpr{N,T}))=Z(\mnh)\,\big(\escpr{B(Z),S}-1\big)+\mnh\,Z(\escpr{B(Z),S}).
\]
Equation \eqref{eq:vmnh} together with the second equality in \eqref{eq:relations} implies
\begin{equation}
\label{eq:zmnh}
Z(\mnh)=\escpr{N,T}\,\big(1-\escpr{B(Z),S}\big).
\end{equation}
The derivative $Z(\escpr{B(Z),S})$ was computed in \cite[Lem.~5.5]{rosales}. We have
\[
Z(\escpr{B(Z),S})=4\,\mnh\,\escpr{N,T}\,(1-K-H^2)-2\,\mnh^{-1}\,\escpr{N,T}\,\escpr{B(Z),S}\,\big(1+\escpr{B(Z),S}\big).
\]

On the other hand, note that $D_ZZ=(2H)\,\nuh$. This is a consequence of \eqref{eq:mc2} together with identities $\escpr{D_ZZ,Z}=0$ and $\escpr{D_ZZ,T}=-\escpr{Z,J(Z)}=0$. Therefore, we deduce that
\begin{align}
\nonumber
B(Z)&=\escpr{B(Z),Z}\,Z+\escpr{B(Z),S}\,S=\escpr{N,D_ZZ}\,Z+\escpr{B(Z),S}\,S
\\
\label{eq:guachi}
&=2H\,\mnh\,Z+\escpr{B(Z),S}\,S.
\end{align}
From here, it is straightforward to check that
\[
|B(Z)+S|^2+4\,(K-1)\,\mnh^2=4\,(H^2+K-1)\,\mnh^2+\big(1+\escpr{B(Z),S}\big)^2.
\]

The proof finishes by substituting the previous equalities into \eqref{eq:lu} and simplifying.
\end{proof}

\subsection{Criteria for strong stability}
\label{subsec:criteria}
\noindent

Now, we have all the ingredients necessary to prove the following result ensuring strong stability of the regular set of a CMC surface.

\begin{theorem}
\label{th:criterion}
Let $M$ be a Sasakian sub-Riemannian $3$-manifold. Consider an oriented $C^2$ surface $\Sg$ immersed in $M$ such that $\Sg-\Sg_0$ is $C^3$ and has constant mean curvature $H$. Suppose there is a nowhere vanishing function $\psi\in C^2(\Sg-\Sg_0)$ such that $\psi\,\ele(\psi)\leq 0$ $($resp. $\psi\,\ele(\psi)< 0$$)$. Then, we have $\mathcal{Q}(w,w)\geq 0$ $($resp. $\mathcal{Q}(w,w)>0$$)$ for any $w\in C^1_0(\Sg-\Sg_0)$ with $w\neq 0$. In particular, $\Sg-\Sg_0$ is strongly stable $($resp. strictly stable$)$.
\end{theorem}

\begin{proof}
We denote $q:=|B(Z)+S|^2+4\,(K-1)\,|N_{h}|^2$. Take any function $w\in C_0^1(\Sg-\Sg_0)$. We define $f:=w/\psi$. Clearly $f\in C^1_0(\Sg-\Sg_0)$. From equation \eqref{eq:indexform}, we get
\begin{align}
\nonumber
\mathcal{Q}(w,w)&=\mathcal{Q}(f\psi,f\psi)=\int_{\Sg}\mnh^{-1}
\left\{Z(f\psi)^2-q\,f^2\psi^2\right\}da
\\
\label{eq:exodus}
&=\int_{\Sg}\mnh^{-1}\left\{f^2\,Z(\psi)^2+\psi^2\,Z(f)^2+\psi\,Z(f^2)\,Z(\psi)-q\,f^2\psi^2\right\}da.
\end{align}
On the other hand, we can apply \eqref{eq:ibp} with $u=f^2\psi$ and $v=\psi$, so that we obtain
\begin{align*}
-\int_{\Sg} f^2\,\psi\,\ele(\psi)\,da&=\int_{\Sg}\mnh^{-1}\left\{Z(f^2\psi)\,Z(\psi)-q\,f^2\psi^2\right\}da
\\
&=\mathcal{Q}(w,w)-\int_{\Sg}\mnh^{-1}\psi^2\,Z(f)^2\,da,
\end{align*}
where in the second equality we have used \eqref{eq:exodus}. From the previous formula we conclude that
\begin{equation}
\label{eq:sorpasso}
\mathcal{Q}(w,w)=\int_{\Sg}\mnh^{-1}\left\{\psi^2\,Z(f)^2-\mnh\,\psi\,\ele(\psi)\,f^2\right\}da.
\end{equation}
Hence, the fact that $\psi\,\ele(\psi)\leq 0$ on $\Sg-\Sg_0$ implies that $Q(w,w)\geq 0$. Moreover, if $\psi\,\ele(\psi)<0$ on $\Sg-\Sg_0$ and $Q(w,w)=0$, then \eqref{eq:sorpasso} implies that $f=0$, and so $w=0$.

Finally, since $\Sg-\Sg_0$ is a CMC surface, we deduce by Proposition~\ref{prop:vpstationary} that $\Sg-\Sg_0$ is volume-preserving area-stationary (resp. area-stationary if $H=0$). Take any variation $\varphi$ of $\Sg-\Sg_0$ with velocity vector field $U$. By denoting $w:=\escpr{U,N}$ and applying the second variation formula in \eqref{eq:empty}, we infer that $(A+2HV)''(0)=Q(w,w)\geq 0$. This shows that $\Sg-\Sg_0$ is strongly stable. Moreover, if $\psi\,\ele(\psi)<0$ on $\Sg-\Sg_0$, then equality $(A+2HV)''(0)=0$ yields $w=0$, so that $U$ is everywhere tangent to $\Sg$. This completes the proof.
\end{proof}

As a consequence of Theorem~\ref{th:criterion} and Lemma~\ref{lem:jacobi} we can deduce the next result.

\begin{corollary}
\label{cor:criteria}
Let $M$ be a Sasakian sub-Riemannian $3$-manifold. Consider an oriented $C^2$ surface $\Sg$ immersed in $M$ such that $\Sg-\Sg_0$ is $C^3$ and has constant mean curvature $H$. Suppose  that one of the following conditions hold:
\begin{itemize}
\item[(i)] there is a nowhere vanishing Jacobi function $\psi$ on $\Sg$,
\item[(ii)] there are no vertical points in $\Sg$ $($the function $\escpr{N,T}$ is positive or negative on $\Sg$$)$,
\item[(iii)] $\ele(\mnh)\leq 0$ on $\Sg-\Sg_0$.
\end{itemize}
Then, the regular set $\Sg-\Sg_0$ is strongly stable. Moreover, if $\ele(\mnh)<0$ on $\Sg-\Sg_0$, then $\Sg-\Sg_0$ is strictly stable.
\end{corollary}

\begin{example}[$t$-graphs]
Let $\Om\subeq\mathbb{N}(\kappa)$, $\kappa\leq 0$, be an open set. Consider a graph $\Sg\sub\e$ of the form $t=f(x,y)$, for some function $f\in C^2(\Om)$ which is also $C^3$ off of the singular set. If $\Sg$ has CMC then Corollary~\ref{cor:criteria} (i) gives strong stability of $\Sg-\Sg_0$. In the particular case of minimal graphs in $\mm(0)$ this is also a consequence of a calibration argument, see \cite[Sect.~5]{rr2} and \cite[Sect.~2]{bscv}.
\end{example}

\begin{remark}
In the setting of Carnot groups of arbitrary dimension, Montefalcone~\cite[Lem.~5.5]{montefalcone2} proved Lemma~\ref{lem:jacobi} for any left-invariant vertical field, and used it to obtain in \cite[Sect.~6]{montefalcone2} some criteria for strong/strict stability of minimal hypersurfaces similar to Corollary~\ref{cor:criteria} (i). On the other hand, Galli showed in \cite[Lem.~5.3]{galli2} that, if a complete minimal surface with empty singular set in the sub-Riemannian Sol manifold satisfies $\escpr{N,T}\leq 0$, then it is strongly stable. As it is shown in Example~\ref{ex:vertical} this fact need not hold in Sasakian sub-Riemannian $3$-manifolds.
\end{remark}

\section{Strongly stable surfaces with empty singular set}
\label{sec:helicoids}

In this section we obtain new examples of complete volume-preserving area-stationary surfaces with empty singular set in the model spaces $\e$. For that, we will produce suitable deformations of CMC vertical surfaces, see the Introduction for a motivation of our construction. Then, we will use Corollary~\ref{cor:criteria} (iii) to answer positively an open question from \cite[Re.~6.10]{rosales} regarding the existence of complete stable non-vertical surfaces in $\mm(-1)$ having empty singular set and constant mean curvature $H\in[0,1)$.

Let $\Ga:\rr\to\e$ be the vertical axis of $\e$. For any $\eps\in\rr$ we denote $X(\eps):=X_{\Ga(\eps)}$, $Y(\eps):=Y_{\Ga(\eps)}$ and $T(\eps):=T_{\Ga(\eps)}$, where $X$, $Y$ and $T$ are the vector fields introduced in Section~\ref{subsec:model}. Given a $C^n$ function $\sigma:\rr\to\rr$ with $n\in\nn$, we define a $C^n$ unit horizontal vector field along $\Ga$ by $U(\eps):=\cos\sg(\eps)\,X(\eps)+\sin\sg(\eps)\,Y(\eps)$. If $\kappa=-1,0$, then it is easy to check that $\sg(\eps)=\theta(\eps)+2\kappa\eps$, where $\theta(\eps)$ is the Euclidean angle between $U(\eps)$ and $\ptl_x$ in the plane generated by $\{\ptl_x,\ptl_y\}$. In the case $\kappa=1$, we have $\sg(\eps)=\theta(\eps)+\eps$, where $\theta(\eps)$ is the Euclidean angle between $U(\eps)$ and $\ptl_{x_2}$ in the plane generated by $\{\ptl_{x_2},\ptl_{y_2}\}$. Sometimes we will work with the angle function $\theta(\eps)$ instead of $\sg(\eps)$ since it provides a clearer geometric interpretation of our construction.

Fix a number $\la\in\rr$. For any $\eps\in\rr$, let $\ga_\eps:\rr\to\e$ be the CC-geodesic of curvature $\la$ with $\ga_\eps(0)=\Ga(\eps)$ and $\dot{\ga}_\eps(0)=U(\eps)$. The \emph{associated one-parameter flow} is the map $F:\rr^2\to\e$ given by $F(\eps,s):=\ga_\eps(s)$. We denote
\begin{equation}
\label{eq:sgla}
\Sg_{\la,\sg}:=F(\rr^2)=\{\ga_\eps(s)\,;\,(\eps,s)\in\rr^2\}.
\end{equation}
Note that $\Sg_{\la,\sg}$ is uniquely determined by the CC-geodesic $\ga_0$ and a one-parameter family of vertical screw motions associated to the angle function $\sg$.

The following is the main result of this section.

\begin{theorem}
\label{th:main1}
The set $\Sg_{\la,\sg}$ defined in \eqref{eq:sgla} satisfies the following properties:
\begin{itemize}
\item[(i)] If $\la^2+\kappa\leq 0$, then $\Sg_{\la,\sg}$ is a complete surface of class $C^n$ immersed in $\e$. Moreover, $\Sg_{\la,\sg}$ has empty singular set if and only if the angle function $\theta$ is nondecreasing.
\item[(ii)] If $\la^2+\kappa>0$, then $\Sg_{\la,\sg}$ is a complete surface of class $C^n$ immersed in $\e$ if and only if $\theta'\neq -(\la^2+\kappa)$ along $\Ga$ when $\kappa=-1,0$, or $\theta'\neq -\la^2$ along $\Ga$ when $\kappa=1$. In such a case, the singular set of $\Sg_{\la,\sg}$ is empty if and only if $\theta'>-(\la^2+\kappa)$ along $\Ga$ when $\kappa=-1,0$, or $\theta'>-\la^2$ along $\Ga$ when $\kappa=1$.
\item[(iii)] $\Sg_{\la,\sg}$ is a vertical surface if and only if $\kappa=-1,0$ and $\theta$ is constant along $\Ga$, or $\kappa=1$ and $\theta$ is a translation of the parameter along $\Ga$.
\item[(iv)] If $\Sg_{\la,\sg}$ is an immersed surface with empty singular set, then there is a Riemannian unit normal $N$ such that any complete CC-geodesic $\ga_\eps:\rr\to\e$ is a characteristic curve of $\Sg_{\la,\sg}$. Moreover, $\Sg_{\la,\sg}$ is a volume-preserving area-stationary surface with constant mean curvature $\la$ with respect to $N$.
\item[(v)] Suppose $\kappa=-1$ and $\la^2<1$. If the angle function $\theta$ is $C^3$ and satisfies $0\leq\theta'\leq 1-\la^2$ $($resp. $0\leq\theta'<1-\la^2$$)$ along $\Ga$, then $\Sg_{\la,\sg}$ is a complete strongly stable $($resp. strictly stable$)$ surface in $\mm(-1)$ with empty singular set.
\item[(vi)] Any complete, connected, $C^1$ orientable surface with empty singular set and constant mean curvature which contains a vertical line is congruent to some surface $\Sg_{\la,\sg}$.
\end{itemize}
\end{theorem}

\begin{proof}
The flow $F(\eps,s)=\ga_\eps(s)$ is $C^n$ since $U(\eps)$ is $C^n$ and, for fixed $\la\in\rr$, the solutions of the geodesic equation \eqref{eq:geoeq} depends differentiably on the initial data. Note that $(\ptl F/\ptl s)(\eps,s)=\dot{\ga}_\eps(s)$, which is a horizontal vector. We denote $V_\eps(s):=(\ptl F/\ptl\eps)(\eps,s)$. From Lemma~\ref{lem:ccjacobi} (i) it follows that $V_\eps$ is a $C^\infty$ vector field along $\ga_\eps$ with $[\dot{\ga}_\eps,V_\eps]=0$. It is also clear that $V_\eps(0)=\dot{\Ga}(\eps)=T(\eps)$. Let $v_\eps:=\escpr{V_\eps,T}$. From Lemma~\ref{lem:ccjacobi} (ii) we deduce that the expression of $V_\eps$ with respect to the orthonormal frame $\{\dot{\ga}_\eps,J(\dot{\ga}_\eps),T\}$ is given by
\begin{equation}
\label{eq:vtheta2}
V_\eps=\lambda\,(1-v_\eps)\,\dot{\ga}_\eps+(v'_\eps/2)\,J(\dot{\ga}_\eps)+v_\eps\,T.
\end{equation}
Evaluating the previous expression at $s=0$ gives $v_\eps'(0)=0$ since $v_\eps(0)=1$. By taking covariant derivatives along $\ga_\eps$ we have
\[
V_\eps'(0)=(1+v_\eps''(0)/2)\,J(U(\eps)),
\]
since $D_{\dot{\ga}_\eps(0)}T=J(\dot{\ga}_\eps(0))=J(U(\eps))$. On the other hand, the equality $[\dot{\ga}_\eps,V_\eps]=0$ implies $V_\eps'=D_{\dot{\ga}_\eps}V_\eps=D_{V_\eps}\dot{\ga}_\eps$ along $\ga_\eps$. As a consequence
\begin{align*}
V_\eps'(0)&=\frac{d}{d\eps}\bigg|_\eps\,\,\dot{\ga}_\eps(0)=\frac{d}{d\eps}\bigg|_\eps\,U(\eps)=\frac{d}{d\eps}\bigg|_\eps\,\big(\!\cos \sigma(\eps)\,X(\eps)+\sin
\sigma(\eps)\,Y(\eps)\big)
\\
&=-\sigma'(\eps)\,\sin\sigma(\eps)\,X(\eps)+\cos\sigma(\eps)\,X'(\eps)
+\sigma'(\eps)\,\cos\sigma(\eps)\,Y(\eps)+\sin\sigma(\eps)\,Y'(\eps).
\end{align*}
Note that
\begin{align*}
X'&=D_{\dot{\Ga}}X=D_{T}X=D_{X}T+[T,X]=J(X)-[X,T]=(1-2\kappa)\,Y,
\\
Y'&=D_{\dot{\Ga}}Y=D_{T}Y=D_{Y}T+[T,Y]=J(Y)-[Y,T]=-(1-2\kappa)\,X,
\end{align*}
where we have used \eqref{eq:lb}. By substituting into the previous equation, we infer
\[
V_\eps'(0)=(\sigma'(\eps)+1-2\kappa)\,J(U(\eps)).
\]
Now, if we compare the two expressions for $V_\eps'(0)$ above, then we get $v''_\eps(0)=2\sigma'(\eps)-4\kappa$.

Once we know that $v_\eps(0)=1$, $v_\eps'(0)=0$ and $v''_\eps(0)=2\sigma'(\eps)-4\kappa$, we can apply Lemma~\ref{lem:ccjacobi} (iii) to obtain an explicit expression for $v_\eps$. We denote $\tau:=4\,(\lambda^2+\kappa)$ and distinguish three cases.

\emph{Case 1.} If $\tau<0$, then $\kappa=-1$, and we have
\begin{equation*}
\label{eq:tau<0}
v_\eps(s)=\frac{2\,\sg'(\eps)+4}{\tau}\,\big(1-\cosh(\sqrt{-\tau}\,s)\big)+1=
\frac{2\,\theta'(\eps)}{\tau}\,\big(1-\cosh(\sqrt{-\tau}\,s)\big)+1,
\end{equation*}
since $\sg(\eps)=\theta(\eps)-2\eps$. In particular, $v_\eps$ and $v_\eps'$ never vanish simultaneously. By \eqref{eq:vtheta2} it follows that $V_\eps$ cannot be proportional to $\dot{\ga}_\eps$ and so, the differential of the flow $F:\rr^2\to\e$ has always rank two. This shows that $\Sg_{\la,\sg}$ is a $C^n$  surface immersed in $\e$. Moreover, the singular set of $\Sg_{\la,\sg}$ consists of the points $\ga_\eps(s)$ such that $v_\eps(s)=0$. Clearly the condition $\theta'(\eps)\geq 0$ implies that $v_\eps(s)\geq 1$ for any $s\in\rr$. Indeed, if $\theta'(\eps_0)<0$, then $v_{\eps_0}(s_0)=0$ for some $s_0\in\rr$ since $-2\theta'(\eps_0)/\tau<0$ and the function $\omega:\rr\to\rr$ given by $\omega(s):=1-\cosh(\sqrt{-\tau}\,s)$ satisfies $\omega(\rr)=(-\infty,0]$. From this analysis we conclude that $\Sg_{\la,\sg}$ has empty singular set if and only if $\theta'(\eps)\geq 0$, for any $\eps\in\rr$.

\emph{Case 2.} If $\tau=0$, then $\kappa\leq 0$ and
\[
v_\eps(s)=(\sg'(\eps)-2\kappa)\,s^2+1=\theta'(\eps)\,s^2+1,
\]
since $\sg(\eps)=\theta(\eps)+2\kappa\eps$. We can proceed as in the previous case to get that $\Sg_{\la,\sg}$ is a $C^n$ surface immersed in $\e$, and that it has empty singular set if and only if $\theta'(\eps)\geq 0$ for any $\eps\in\rr$. This proves statement (i) of the theorem.

\emph{Case 3.} If $\tau>0$, then we get
\begin{align*}
v_\eps(s)&=\frac{2\,\sg'(\eps)-4\kappa}{\tau}\,\big(1-\cos(\sqrt{\tau}\,s)\big)+1
\\
\nonumber
&=
\begin{cases}
\frac{2\theta'(\eps)}{\tau}\,\big(1-\cos(\sqrt{\tau}\,s)\big)+1, \hspace{0.930cm}\text{ if } \kappa=-1,0,
\vspace{0,1cm}
\\
\frac{2(\theta'(\eps)-1)}{\tau}\,\big(1-\cos(\sqrt{\tau}\,s)\big)+1, \quad\text{ if } \kappa=1.
\end{cases}
\end{align*}
By equation \eqref{eq:vtheta2}, the map $F$ fails to be an immersion at $(\eps,s)\in\rr^2$ if and only if $v_\eps(s)=v'_\eps(s)=0$. This is equivalent to that $\theta'(\eps)=-(\la^2+\kappa)$ and $s=(2m+1)\pi/\sqrt{\tau}$ ($m\in\mathbb{Z}$) if $\kappa=-1,0$, or $\theta'(\eps)=-\la^2$ and $s=(2m+1)\pi/\sqrt{\tau}$ ($m\in\mathbb{Z}$) if $\kappa=1$. If $F$ is an immersion, then the singular set of $\Sg_{\la,\sg}$ consists of the points $\ga_\eps(s)$ such that $v_\eps(s)=0$. Denoting $m(\eps):=(2\sg'(\eps)-4\kappa)/\tau$, and taking into account that the function $\omega:\rr\to\rr$ defined by $\omega(s):=1-\cos(\sqrt{\tau}\,s)$ satisfies $\omega(\rr)=[0,2]$, it is straightforward to check that $v_\eps(s)\neq 0$ for any $s\in\rr$ if and only if $m(\eps)>-1/2$. Hence, the claim in (ii) easily follows since $\sg'=\theta'+2\kappa$ when $\kappa=-1,0$ and $\sg'=\theta'+1$ when $\kappa=1$.

Suppose now that $\Sg_{\la,\sg}$ is an immersed surface. We take a point $p=\ga_\eps(s)\in\Sg_{\la,\sg}$. Since the tangent plane at $p$ is generated by $\{\dot{\ga}_\eps(s),V_\eps(s)\}$ and the functions $v_\eps(s)$, $v_\eps'(s)$ do not vanish simultaneously, it follows from \eqref{eq:vtheta2} that $p$ is a vertical point of $\Sg_{\la,\sg}$ if and only if $v_\eps'(s)=0$. This allows to describe the set of vertical points after an explicit computation of $v_\eps'(s)$. If $\tau\leq 0$ this set is the union of $\Ga$ together with the CC-geodesics $\ga_\eps$ for which $\theta'(\eps)=0$. If $\tau>0$, then $\ga_\eps(s)$ is vertical if and only if $s=m\pi/\sqrt{\tau}$ with $m\in\mathbb{Z}$, or $\theta'(\eps)=0$ and $\kappa=-1,0$, or $\theta'(\eps)=1$ and $\kappa=1$. From here we deduce statement (iii).

In order to prove (iv) we define the $C^{n-1}$ vector field
\begin{equation}
\label{eq:normal}
N(\eps,s):=\frac{-v_\eps(s)\,J(\dot{\ga}_\eps(s))+(v_\eps'(s)/2)\,T}{\sqrt{v_\eps(s)^2+(v'_\eps(s)/2)^2}},
\end{equation}
which clearly provides a unit normal along $\Sg_{\la,\sg}$. Since we assume that the singular set is empty, then $v_\eps(s)>0$ for any $(\eps,s)\in\rr^2$. From the definition of $\nuh$ and $Z$ in \eqref{eq:nuh}, we get the equalities
\[
\nuh(\eps,s)=-J(\dot{\ga}_\eps(s)), \quad Z(\eps,s)=\dot{\ga}_\eps(s),
\]
and so, any complete CC-geodesic $\ga_\eps$ is a characteristic curve of $\Sg_{\la,\sg}$. By equations \eqref{eq:mc} and \eqref{eq:geoeq} we infer
\[
2H=\escpr{D_ZZ,\nuh}=\escpr{\dot{\ga}_\eps',\nuh}=-2\la\,\escpr{J(\dot{\ga}_\eps),\nuh}=2\la,
\]
from which $\Sg_{\la,\sg}$ has constant mean curvature $\la$ with respect to $N$. We conclude that $\Sg_{\la,\sg}$ is volume-preserving area-stationary by Proposition~\ref{prop:vpstationary}.

Let us prove (v). If we assume $\kappa=-1$, $\la^2<1$ and $\theta'\geq 0$, then statement (i) implies that $\Sg_{\la,\sg}$ is a complete $C^3$ surface immersed in $\mm(-1)$ with empty singular set. By statement (iv) the surface $\Sg_{\la,\sg}$ is volume-preserving area-stationary with constant mean curvature $\la$, and any CC-geodesic $\ga_\eps$ is a characteristic curve. To deduce strong (resp. strict) stability it suffices, by Corollary~\ref{cor:criteria} (iii), to check that $\ele(\mnh)\leq 0$ (resp. $\ele(\mnh)<0$) along $\Sg_{\la,\sg}$. If we denote $\tau:=4\,(\la^2-1)$, then we have the equality
\begin{equation}
\label{eq:stable1}
\ele(\mnh)=4\,\mnh^{-2}\,\escpr{B(Z),S}+\tau-4,
\end{equation}
which follows from \cite[Lem.~6.5]{rosales}. Next, we will obtain an explicit expression for $\ele(\mnh)$ depending on the angle function $\theta$. Equation~\eqref{eq:normal} gives us
\begin{equation*}
\label{eq:pistola1}
\escpr{N,T}=\frac{v'_\eps}{\sqrt{4\,v^2_\eps+(v_\eps')^2}}, \quad
\mnh=\frac{2\,v_\eps}{\sqrt{4\,v^2_\eps+(v_\eps')^2}}.
\end{equation*}
By taking into account \eqref{eq:znt} and that any $\ga_\eps$ is a characteristic curves of $\Sg_{\la,\sg}$, we infer
\begin{equation*}
\label{eq:pistola2}
\escpr{B(Z),S}=\frac{Z(\escpr{N,T})}{\mnh}+1=\frac{2\,v_\eps\,v_\eps''+4\,v_\eps^2-(v_\eps')^2}{4\,v_\eps^2+(v_\eps')^2}.
\end{equation*}
By substituting the previous equalities into \eqref{eq:stable1}, we deduce
\[
v_\eps^2\,\ele(\mnh)=2\,v_\eps\,v_\eps''-(v_\eps')^2+\tau\,v_\eps^2.
\]
By differentiating above with respect to $s$, and having in mind the differential equation in Lemma~\ref{lem:ccjacobi} (iii), we get that $2v_\eps v_\eps''-(v_\eps')^2+\tau v_\eps^2$ is constant along $\ga_\eps$. By evaluating at $s=0$ and using that $v_\eps(0)=1$, $v_\eps'(0)=0$ and $v''_\eps(0)=2\sg'(\eps)+4=2\theta'(\eps)$, we conclude that
\[
v_\eps^2\,\ele(\mnh)=4\,\big(\theta'(\eps)+\la^2-1\big).
\]
Therefore $\ele(\mnh)\leq 0$ (resp. $\ele(\mnh)<0$) if and only if $\theta'\leq 1-\la^2$ (resp. $\theta'<1-\la^2$) along $\Ga$.

It remains to prove (vi). Consider a complete, connected, orientable surface $\Sg$ immersed in $\e$ with $\Sg_0=\emptyset$ and constant mean curvature $\la$. Suppose that $\Sg$ contains a vertical line. After a left or right translation in $\e$ the surface $\Sg$ is congruent to a surface $\Sg'$ satisfying the same properties and containing the vertical axis. From Proposition~\ref{prop:vpstationary} the characteristic curve $\ga_\eps$ of $\Sg'$ passing through $\Ga(\eps)$ is a CC-geodesic of curvature $\la$ with $\dot{\ga}_\eps(0)=\cos\sg(\eps)\,X(\eps)+\sin\sg(\eps)\,Y(\eps)$, for some angle function $\sg(\eps)$. Since $\Sg'$ is complete and connected with $\Sg'_0=\emptyset$ then we obtain $\Sg'=\Sg_{\la,\sg}$. This shows that $\Sg$ is congruent to $\Sg_{\la,\sg}$, as we claimed.
\end{proof}

We finish this section with some examples and remarks concerning the previous theorem.

\begin{example}[Surfaces $\Sg_{\la,\sg}$ in $\mm(0)$]
It is well known that a CC-geode\-sic of vanishing curvature in the Heisenberg group $\mm(0)$ is an affine parameterization of a horizontal straight line. Thus, for any angle function $\sg$, the surface $\Sg_{0,\sg}$ is the union of a family of horizontal lines contained in parallel planes. Since $\Ga(\eps)=(0,0,\eps)$ and $U(\eps)=(\cos\sg(\eps),\sin\sg(\eps),0)$, it follows that
\[
\Sg_{0,\sg}=\{\Ga(\eps)+s\,U(\eps)\,;\,(\eps,s)\in\rr^2\}=\{(x,y,t)\in\rr^3\,;\,x\,\sin\sg(t)-y\,\cos\sg(t)=0\}.
\]
Observe that, for $\sg(\eps)=C\eps$ with $C>0$, we get a right handed helicoid in $\rr^3$. More generally, any $\Sg_{0,\sg}$ is one of the embedded helicoid type minimal surfaces in $\mm(0)$ described by Cheng and Hwang~\cite[Thm.~B]{ch}. The fact that the singular set of $\Sg_{0,\sg}$ is empty if and only if $\sg'\geq 0$ along $\Ga$ is contained in \cite[Thm.~C]{ch}. We also note that the entire graphical strips defined by Danielli, Garofalo, Nhieu and Pauls~\cite[Def.~1.3]{dgnp} in relation to the Bernstein problem in $\mm(0)$ coincide with the surfaces $\Sg_{0,\sg}$ having empty singular set and satisfying $\cos\sg\neq 0$ along $\Ga$ or $\sin\sg\neq 0$ along $\Ga$.

To the best of our knowledge the surfaces $\Sg_{\la,\sg}$ such that $\la\neq 0$ and $\sg$ is a non-constant function with $\sg'\geq 0$ provide new examples of complete volume-preserving area-stationary surfaces in $\mm(0)$ with empty singular set. These surfaces are never compact since they contain $\Ga$. In general the surfaces $\Sg_{\la,\sg}$ are not embedded. From \cite[Cor.~6.9]{rosales} the unique stable surfaces $\Sg_{\la,\sg}$ of class $C^2$ with empty singular set are vertical planes. By a recent work of Galli and Ritor\'e~\cite{galli-ritore3} the vertical planes are also the only strongly stable surfaces $\Sg_{0,\sg}$ of class $C^1$ in $\mm(0)$ with empty singular set.
\end{example}

\begin{example}[Surfaces $\Sg_{\la,\sg}$ in $\mm(-1)$]
\label{ex:minhyper}
Let $p\in\mm(-1)$ be a point on $\Ga$, and $v\in\mathcal{H}_p$ a unit vector. It follows easily from \eqref{eq:geoeq}, see also \cite[Lem.~3.1]{rosales}, that the CC-geode\-sic $\ga$ in $\mm(-1)$ of vanishing curvature with $\ga(0)=p$ and $\dot{\ga}(0)=v$ is given by $\ga(s)=p+\tanh(s)\,v$, for any $s\in\rr$. Hence, for any angle function $\sg$, we have
\begin{equation}
\label{eq:tq}
\Sg_{0,\sg}=\{\Ga(\eps)+\tanh(s)\,U(\eps)\,;\,(\eps,s)\in\rr^2\}=\{(x,y,t)\in\mm(-1)\,;\,x\,\sin\theta(t)-y\,\cos\theta(t)=0\}.
\end{equation}
From Theorem~\ref{th:main1} (v), for any $\theta\in C^3$ with $0\leq\theta'\leq 1$ the surface $\Sg_{0,\sg}$ is strongly stable and has empty singular set. This is a remarkable difference with respect to $\mm(0)$, where the vertical planes are the only complete strongly stable minimal surfaces with empty singular set, see \cite{hrr}, \cite{dgnp-stable}, \cite{galli-ritore3}.

We can produce strongly stable non-minimal surfaces in $\mm(-1)$ as follows. Choose numbers $\la\in (-1,1)$ and $C\in (0,1-\la^2]$. Consider the surface $\Sg_{\la,\sg}$ where $\sg(\eps)=(C-2)\eps$ along $\Ga$. This is an embedded right handed minimal helicoid when $\la=0$. Note that $\theta(\eps)=C\eps$, so that Theorem~\ref{th:main1} implies that $\Sg_{\la,\sg}$ is a $C^\infty$ complete, orientable, strongly stable, non-vertical surface in $\mm(-1)$ with constant mean curvature $\la$ and empty singular set.
\end{example}

\begin{example}[The Bernstein problem in $\mm(-1)$]
\label{ex:bernstein}
The classical Bernstein problem studies entire minimal graphs in Euclidean space. In $\rr^3$ it is well-known that the planes are the unique solutions to this problem. In the Heisenberg group $\mm(0)$ it is possible to find entire minimal graphs (Euclidean or intrinsic) different from planes. Indeed, the best Bernstein type result in $\mm(0)$ establishes that an entire strongly stable minimal graph with empty singular set must be a vertical plane, see \cite[Thm.~1.8]{dgnp} and \cite[Thm.~5.3]{bscv}. To see that this rigidity result does not hold in $\mm(-1)$, it suffices to consider a surface $\Sg_{0,\sg}$ in $\mm(-1)$ with angle function $\sg$ so that $\theta\in (-\pi/2,\pi/2)$ and $0\leq\theta'\leq 1$ along $\Ga$ (there is a continuum of these functions). From \eqref{eq:tq} we have
\[
\Sg_{0,\sg}=\{(x,y,t)\in\mm(-1)\,;\,y=x\,\tan\theta(t)\}.
\]
It follows from Theorem~\ref{th:main1} that $\Sg_{0,\sg}$ is a strongly stable entire minimal graph with empty singular set. In the particular case $\theta(\eps)=\arctan(\eps)$ we get the Euclidean graph $y=xt$.
\end{example}

\begin{example}[Surfaces $\Sg_{\la,\sg}$ in $\mm(1)$]
\label{ex:newtori}
By the results of Cheng, Hwang, Malchiodi and Yang~\cite{chmy} a compact connected $C^2$ surface $\Sg$ in $\mm(1)$ with constant mean curvature $\la$ and $\Sg_0=\emptyset$ is topologically a torus. From \cite[Thm.~5.10, Prop.~5.11]{hr1}, if $\Sg$ is vertical or $\la/\sqrt{1+\la^2}\notin\mathbb{Q}$, then $\Sg$ is congruent to a Clifford torus $\mathcal{T}_r:=\s^1(r)\times\s^1(\sqrt{1-r^2})$. Rotationally symmetric CMC tori which are not congruent to $\mathcal{T}_r$ were described in \cite[Thm.~6.4]{hr1}. Let us see how we can produce new examples of CMC tori which are neither congruent to $\mathcal{T}_r$ nor rotationally symmetric.

The vertical axis in $\mm(1)$ is the great circle $\Ga(\eps)=(\cos\eps,\sin\eps,0,0)$. Fix $\la\in\rr$ such that $\la/\sqrt{1+\la^2}\in\mathbb{Q}$. This ensures by \cite[Prop.~3.3]{hr1} that all the CC-geodesics in $\mm(1)$ of curvature $\la$ are embedded circles of the same length. Take an angle function $\sg$ such that $U:=(\cos\sg)\,X+(\sin\sg)\,Y$ is $2\pi$-periodic and $\theta'>-\la^2$ along $\Ga$. From Theorem~\ref{th:main1} the surface $\Sg_{\la,\sg}$ is an immersed torus with constant mean curvature $\la$ and empty singular set. Since $\Ga\sub\Sg_{\la,\sg}$ then $\Sg_{\la,\sg}$ is not rotationally symmetric about $\Ga$. Moreover, if $\theta$ is not a translation along $\Ga$, then $\Sg_{\la,\sg}$ cannot be vertical and so, it is not congruent to $\mathcal{T}_r$. In general these surfaces $\Sg_{\la,\sg}$ are not embedded even if $\la=0$. As a consequence of the stability result in \cite[Cor.~6.9]{rosales} none of these examples is stable.
\end{example}

\begin{remarks}
1. Let $M$ be an arbitrary $3$-dimensional space form of Webster scalar curvature $\kappa$. As was pointed out in \cite[Prop.~2.1]{hr2} there is a surjective local isometry $\Pi:\e\to M$. From here we can define the surfaces $\Sg_{\la,\sg}$ and prove Theorem~\ref{th:main1} in this more general setting.

2. It is natural to ask if other criteria different from Corollary~\ref{cor:criteria} (iii) may be used to deduce strong/strict stability of $\Sg_{\la,\sg}$ when $\kappa=-1$, $\la^2<1$ and $\theta'\geq 0$. Note that the sufficient condition in Corollary~\ref{cor:criteria} (ii) never holds on $\Sg_{\la,\sg}$ since any $p\in\Ga$ is vertical. As to the condition in Remark~\ref{re:trivial}, we can follow the arguments in the proof of Theorem~\ref{th:main1} to show that the function $q:=|B(Z)+S|^2-8\,\mnh^2$ satisfies
\[
\frac{\big(4\,v^2_\eps+(v'_\eps)^2\big)^2}{4\,v_\eps^2}\,q=4\,\big(\theta'(\eps)^2+4\,\theta'(\eps)+\tau\big).
\]
Now, it is easy to conclude that the inequality $q\leq 0$ implies $\theta'\leq1-\la^2$ along $\Ga$, thus recovering the stability condition in Theorem~\ref{th:main1} (v).
\end{remarks}

\section{Strong stability of surfaces with non-empty singular set}
\label{sec:nonempty}

In this section we use our stability criteria to deduce stability properties of complete $C^2$ volume-preserving area-stationary surfaces with non-empty singular set in a $3$-dimensional space form. In this setting there are rigidity results, see \cite[Sect.~4]{hr2} and the references therein, showing that these surfaces are uniquely determined by their singular set and their mean curvature. As was shown in \cite{chmy}, see also \cite[Sect.~5]{galli}, the singular set of a CMC surface of class $C^2$  consists of isolated points and/or curves with non-vanishing tangent vector. This leads us to divide this section in two parts, where we analyze the two cases separately.

\subsection{Surfaces with isolated singular points}
\label{subsec:planes}
\noindent

Let $M$ be a $3$-dimensional space form of Webster scalar curvature $\kappa$. For any point $p\in M$ and any number $\la\geq 0$, we define
\begin{align}
\label{eq:planes}
\pla&:=\big\{\ga_v(s)\,;\,v\in\h_p,\,|v|=1,\,s\geq 0\}, \hspace{2.17cm} \text{ if } \la^2+\kappa\leq 0,
\\
\label{eq:spheres}
\sla(p)&:=\big\{\ga_v(s)\,;\,v\in\h_p,\,|v|=1,\,s\in[0,\pi/\sqrt{\la^2+\kappa}]\}, \ \text{ if } \la^2+\kappa>0.
\end{align}
Here $\ga_v:\rr\to M$ denotes the CC-geodesic of curvature $\la$ with $\ga_v(0)=p$ and $\dot{\ga}_v(0)=v$.

The sets $\pla$ and $\sla(p)$ were studied in \cite[Sect.~3, Sect.~4.2]{hr2}. We proved that they are the unique complete, connected, orientable, $C^2$ surfaces of constant mean curvature $\la$ in $M$ with at least one isolated singular point. The surfaces $\sla(p)$ are generalizations of the well-known Pansu spheres in the first Heisenberg group $\mm(0)$. The surfaces $\pla$ are volume-preserving area-stationary immersed planes, whose singular set is the point $p$, called the \emph{pole} of $\pla$. In $\mm(0)$ any $\pla$ is a horizontal Euclidean plane $\mathcal{P}_0(p)$. In $\mm(1)$ there are no planes $\pla$. In $\mm(-1)$, up to left translations, there is a unique plane $\pla$ (which is embedded) for any $\la\in [0,1]$. All the planes $\pla$ are $C^\infty$ surfaces off of the pole.

The stability properties of $\sla(p)$ were analyzed in \cite[Sect.~5]{hr2}, where it is shown that $\sla(p)$ is stable but not strongly stable. In this section we will prove that all the planes $\pla$ are strictly stable. We first provide some basic facts and computations that will be helpful in the sequel.

\begin{lemma}
\label{lem:comppla}
Consider a plane $\pla$ immersed in a $3$-dimensional space form $M$ of Webster scalar curvature $\kappa$. Let $N$ be the unit normal along $\pla$ with associated mean curvature $H=\la$. Let $\{Z,S\}$ be the orthonormal basis defined in \eqref{eq:nuh} and \eqref{eq:ese}, and $B$ the Riemannian shape operator with respect to $N$. Denote $\mu:=\sqrt{-(\la^2+\kappa)}$. Then, we have:
\begin{itemize}
\item[(i)] the function $\mnh^{-1}$ is locally integrable on $\pla$ with respect to $da$,
\item[(ii)] $\escpr{N,T}>0$ on $\pla$ $($in particular, there are no vertical points along $\pla$$)$,
\item[(iii)] $\escpr{B(Z),S}=(1+\mu^2)\,\mnh^2$ on $\pla-\{p\}$,
\item[(iv)] $|B(Z)+S|^2+4\,(\kappa-1)\,\mnh^2=\big(1-(1+\mu^2)\,\mnh^2\big)^2$ on $\pla-\{p\}$,
\item[(v)] $\escpr{B(S),S}=\la\,\mnh\,\big(1-(1+\mu^2)\,\mnh^2\big)$ on $\pla-\{p\}$.
\end{itemize}
\end{lemma}

\begin{proof}
Take a positive orthonormal basis $\{e_1,e_2\}$ in $\h_p$. Define $F:\rr^2\to M$ by $F(\theta,s):=\ga_\theta(s)$, where $\ga_\theta$ is the CC-geodesic of curvature $\la$ with $\ga_\theta(0)=p$ and $\dot{\ga}_\theta(0)=(\cos\theta)\,e_1+(\sin\theta)\,e_2$. This is a $C^\infty$ map since, for fixed $\la$, the solutions of \eqref{eq:geoeq} depends differentiably on the initial data. We apply Lemma~\ref{lem:ccjacobi} with $\alpha(\theta):=p$ and $U(\theta):=(\cos\theta)\,e_1+(\sin\theta)\,e_2$. Hence the vector field $V_\theta(s):=(\ptl F/\ptl\theta)(\theta,s)$ satisfies $V_\theta(0)=0$, and
\begin{equation}
\label{eq:vtheta}
V_\theta=-(\la\,v_\theta)\,\dot{\ga}_\theta+(v_\theta'/2)\,J(\dot{\ga}_\theta)+v_\theta\,T,
\end{equation}
where the primes $'$ denote derivatives with respect to $s$. Some computations similar to those after \eqref{eq:vtheta2} lead to equalities $v_\theta(0)=v'_\theta(0)=0$ and $v_\theta''(0)=2$. By applying Lemma~\ref{lem:ccjacobi} (iii), we deduce
\begin{equation*}
\label{eq:plav}
v_\theta(s)=v(s):=
\begin{cases}
s^2,\quad\hspace{5.03cm}\text{ if } \tau=0,
\\
\frac{-1}{2\mu^2}\,\big(1-\cosh(2\mu s)\big)=\frac{1}{\mu^2}\sinh^2(\mu s),\quad\text{ if } \tau<0.
\end{cases}
\end{equation*}
Clearly $v(s)>0$ for any $s>0$. Since $(\ptl F/\ptl s)(\theta,s)=\dot{\ga}_\theta(s)$, which is a horizontal vector, then the map $F:[0,2\pi]\times\rr^+\to M$ is a $C^\infty$ immersion with $F([0,2\pi]\times\rr^+)=\pla-\{p\}$. Let us see that
\begin{equation}
\label{eq:plan}
N=\frac{-v\,J(\dot{\ga}_\theta)+(v'/2)\,T}{\sqrt{v^2+(v'/2)^2}},\quad\text{on } \pla-\{p\}.
\end{equation}
It is clear that the right hand side above defines a unit normal along $\pla-\{p\}$. By using \eqref{eq:nuh} we have $\nuh=-J(\dot{\ga}_\theta)$ and $Z=\dot{\ga}_\theta$. Hence any $\ga_\theta(s)$ with $s>0$ is a characteristic curve of $\pla$ for this normal.  By equations \eqref{eq:mc} and \eqref{eq:geoeq} we infer that the mean curvature $H$ of $\pla$ with respect to this normal equals $\la$, which proves the claim. From \eqref{eq:plan} we deduce
\begin{equation}
\label{eq:ss}
\escpr{N,T}=\frac{v'}{\sqrt{4\,v^2+(v')^2}}, \quad
\mnh=\frac{2\,v}{\sqrt{4\,v^2+(v')^2}}.
\end{equation}
On the other hand, note that
\[
da=\sqrt{|V_\theta|^2-\escpr{V_\theta,\dot{\ga}_\theta}^2} \ d\theta\,ds=\frac{\sqrt{4\,v^2+(v')^2}}{2} \ d\theta\,ds,
\]
and so
\[
\mnh^{-1}\,da=\frac{4\,v^2+(v')^2}{4\,v} \ d\theta\,ds=
\begin{cases}
(s^2+1)\,d\theta\,ds,\quad\hspace{2.9cm}\text{ if } \tau=0,
\\
\big(\frac{1}{\mu^2}\sinh^2(\mu s)+\cosh^2(\mu s)\big)\,d\theta\,ds,\quad\text{ if } \tau<0.
\end{cases}
\]
This proves statement (i). The inequality $\escpr{N,T}>0$ on $\pla-\{p\}$ comes from the first identity in \eqref{eq:ss} since $v'>0$ on $\rr^+$. We also have $N_p=T_p$ since $p$ is a singular point. This proves (ii).

By taking into account \eqref{eq:znt}, \eqref{eq:ss}, and that any $\ga_\theta$ is a characteristic curve of $\pla$, we obtain
\begin{equation}
\label{eq:mm}
\escpr{B(Z),S}=\frac{Z(\escpr{N,T})}{\mnh}+1=\frac{2\,v\,v''+4\,v^2-(v')^2}{4\,v^2+(v')^2}.
\end{equation}
On the other hand, we have equality $2vv''+4v^2-(v')^2=4\,(1+\mu^2)\,v^2$, which follows by differentiating and using the identity $v'''-4\mu^2v'=0$ in Lemma~\ref{lem:ccjacobi} (iii). Thus, the second equality in \eqref{eq:ss} implies (iii). Observe also that $\escpr{B(Z),Z}=2\la\,\mnh$ by \eqref{eq:guachi}. After some computations, we get
\[
|B(Z)+S|^2+4\,(\kappa-1)\,\mnh^2=-4\,(1+\mu^2)\,\mnh^2+\big(1+\escpr{B(Z),S}
\big)^2=\big(1-(1+\mu^2)\,\mnh^2\big)^2,
\]
where in the second equality we have employed statement (iii). This proves (iv).

Finally, we prove (v). From \eqref{eq:vnt} we obtain $S(\escpr{N,T})=\mnh\,\escpr{B(S),S}$ since $J(S)=\escpr{N,T}\,Z$ and $T^\top=-\mnh\,S$. By \eqref{eq:vtheta}, \eqref{eq:ss} and equality $\nu_h=-J(\dot{\ga}_\theta)$, we infer that
\begin{equation*}
\label{eq:esepolar}
S=\frac{-2}{\sqrt{4\,v^2+(v')^2}}\,V_\theta-\la\,\mnh\,Z.
\end{equation*}
Having in mind \eqref{eq:mm} and that $\escpr{N,T}$ does not depend on $\theta$, we deduce
\[
\escpr{B(S),S}=\frac{S(\escpr{N,T})}{\mnh}=-\la\,Z(\escpr{N,T})=\la\,\mnh\,\big(1-\escpr{B(Z),S}\big),
\]
and the claim follows since $\escpr{B(Z),S}=(1+\mu^2)\,\mnh^2$.
\end{proof}

In the next step we simplify the second variation formula given in \eqref{eq:2ndvar} for admissible variations of $\pla$. Indeed, the following result shows that $(A+2\la V)''(0)=Q(w,w)$ even for variations moving the pole of $\pla$. We must remark that, since the function $\mnh^{-1}$ is locally integrable on $\pla$, the family of admissible variations of $\pla$ moving the pole is very large, see \cite[App.~B]{hr2}.

\begin{proposition}
\label{prop:2ndpla}
Let $M$ be a $3$-dimensional space form of Webster scalar curvature $\kappa$. Consider a plane $\pla$ immersed in $M$ for some $p\in M$ and $\la\geq 0$ with $\la^2+\kappa\leq 0$. Let $\varphi:I\times\pla\to M$ be an admissible variation of class $C^3$ off of the pole. Denote $w:=\escpr{U,N}$, where $U$ is the velocity vector field and $N$ is the unit normal with associated mean curvature $H=\la$. Then, the functional $A+2\la V$ is twice differentiable at $s=0$, and we have
\[
(A+2\la V)''(0)=\mathcal{Q}(w,w),
\]
where $\mathcal{Q}$ is the index form defined in \eqref{eq:indexform}.
\end{proposition}

Note that, by statements (i) and (iv) in Lemma~\ref{lem:comppla}, all the terms in $\mathcal{Q}(u,v)$ are locally integrable whenever $u,v\in C^1(\pla)$ and at least one of them has compact support on $\pla$.

\begin{proof}[Proof of Proposition~\ref{prop:2ndpla}]
By equation \eqref{eq:2ndvar} it suffices to check that $\divv_{\pla}G$ is integrable with respect to $da$ and has vanishing integral. This can be done by reproducing the arguments for the spherical surfaces $\sla(p)$, see the proof of Thm.~5.2 in \cite[App.~A]{hr2}. This requires Lemma~\ref{lem:comppla} and the generalized divergence theorem for $\pla$ in Lemma~\ref{lem:divpla} below.
\end{proof}

\begin{lemma}
\label{lem:divpla}
Let $U$ be a bounded and tangent $C^1$ vector field on $\pla-\{p\}$, vanishing off of a compact set of $\pla$, and such that $\divv_{\pla} U$ is integrable with respect to $da$. Then, we have
\[
\int_{\pla}\divv_{\pla}U\,da=0.
\]
\end{lemma}

\begin{proof}
The result follows from an approximation argument similar to the one in  \cite[Lem.~7.4]{hr2}.
\end{proof}

Now, we begin to discuss the stability properties of $\pla$. By Proposition~\ref{prop:2ndpla} and equation \eqref{eq:indexform}, it is natural to check first if $q:=|B(Z)+S|^2+4\,(\kappa-1)\,\mnh^2$ is a non-positive function. Unfortunately Lemma~\ref{lem:comppla} (iv) gives $q\geq 0$ and $q\neq 0$ on $\pla$, so that we cannot obtain strong stability from this method. However, we can use Lemma~\ref{lem:comppla} (ii) and Corollary~\ref{cor:criteria} (ii) to deduce that $\pla-\{p\}$ is strongly stable. The technical difficulty in proving that $\pla$ is strongly stable arises when we deform $\pla$ by \emph{variations moving the pole}. If we analyze the proof of Theorem~\ref{th:criterion} with $\psi=\escpr{N,T}$ and $w\in C_0^1(\pla)$, then we discover that this difficulty is solved provided equality \eqref{eq:ibp} still holds for $u=(w^2/\psi^2)\,\psi$ and $v=\psi$ (these functions has no compact support on $\pla-\{p\}$ if $w(p)\neq 0$). This is established in the next result.

\begin{lemma}
\label{lem:newibp}
Let $M$ be a $3$-dimensional space form of Webster scalar curvature $\kappa$. Consider a plane $\pla$ immersed in $M$ for some $p\in M$ and $\la\geq 0$ with $\la^2+\kappa\leq 0$. If $\mathcal{Q}$ is the index form of $\pla$ and we denote $\psi:=\escpr{N,T}$, then $\mathcal{Q}(u,\psi)=0$ for any $u\in C^1_0(\pla)$.
\end{lemma}

\begin{proof}
We follow the proof of \cite[Prop.~3.14]{hrr}. Let $q:=|B(Z)+S|^2+4\,(\kappa-1)\,\mnh^2$. We have
\begin{equation}
\label{eq:m1}
0=\ele(\psi)=\divv_{\pla}\big(\mnh^{-1}Z(\psi)\,Z\big)+\mnh^{-1}q\,\psi.
\end{equation}
The first equality comes from Lemma~\ref{lem:jacobi}. The second one, which is valid for any $\psi\in C^2(\pla)$, is a straightforward computation by using the definition of $\ele(\psi)$ in \eqref{eq:lu} together with \cite[Lem.~5.5]{rosales} and equation \eqref{eq:zmnh}. We define a tangent $C^1$ vector field on $\pla-\{p\}$ by
\[
U:=\mnh^{-1}Z(\psi)\,u\,Z=\big(\escpr{B(Z),S}-1\big)\,u\,Z,
\]
where in the second equality we have taken into account \eqref{eq:znt}. We know by Lemma~\ref{lem:comppla} (iii) that $\escpr{B(Z),S}=(1+\mu^2)\,\mnh^2$ on $\pla-\{p\}$. Since $u\in C^1_0(\pla)$ we deduce that $U$ is bounded and vanishes off of a compact set of $\pla$. From \eqref{eq:m1}, it follows that
\begin{align*}
\divv_{\pla}U&=u\,\divv_{\pla}\big(\mnh^{-1}Z(\psi)\,Z\big)+\mnh^{-1}\,Z(\psi)\,Z(u)
\\
&=\mnh^{-1}Z(\psi)\,Z(u)-\mnh^{-1}\,q\,\psi\,u
\\
&=\big((1+\mu^2)\,\mnh^2-1\big)\,Z(u)-\mnh^{-1}q\,\psi\,u,
\end{align*}
which is an integrable function with respect to $da$ by statements (i) and (iv) in Lemma~\ref{lem:comppla}. Finally, we apply Lemma~\ref{lem:divpla} to get
\[
0=\int_{\pla}\divv_{\pla}U\,da=\int_{\pla}\mnh^{-1}\,\{Z(\psi)\,Z(u)-q\,\psi\,u\}\,da=\mathcal{Q}(u,\psi),
\]
and the lemma is proved.
\end{proof}

Now, we are ready to establish the main result of this section.

\begin{theorem}
\label{th:stabilitypla}
Let $M$ be a $3$-dimensional space form of Webster scalar curvature $\kappa$. Consider a plane $\pla$ immersed in $M$ for some $p\in M$ and $\la\geq 0$ with $\la^2+\kappa\leq 0$. Then, the index form satisfies $\mathcal{Q}(w,w)>0$, for any $w\in C_0^1(\pla)$ with $w\neq 0$. Hence, $\pla$ is strictly stable under admissible variations.
\end{theorem}

\begin{proof}
Let $N$ be the unit normal on $\pla$ with mean curvature $H=\la$. We denote $q:=|B(Z)+S|^2+4\,(\kappa-1)\,|N_{h}|^2$ and $\psi:=\escpr{N,T}$. We know that $\psi>0$ on $\pla$ by Lemma~\ref{lem:comppla} (ii). Take $w\in C_0^1(\pla)$ and define $f:=w/\psi$, which is in $C^1_0(\pla)$. As in the proof of Theorem~\ref{th:criterion} we get
\begin{equation}
\label{eq:exoduscopia}
\mathcal{Q}(w,w)=\int_{\pla}\mnh^{-1}\left\{f^2\,Z(\psi)^2+\psi^2\,Z(f)^2+\psi\,Z(f^2)\,Z(\psi)-q\,f^2\psi^2\right\}da.
\end{equation}
Now, we apply Lemma~\ref{lem:newibp} with $u:=f^2\psi$ to obtain
\[
0=\mathcal{Q}(u,\psi)=\int_{\pla}\mnh^{-1}\left\{Z(f^2\psi)\,Z(\psi)-q\,f^2\psi^2\right\}da=\mathcal{Q}(w,w)-\int_{\pla}\mnh^{-1}\psi^2\,Z(f)^2\,da,
\]
where in the second equality we have used \eqref{eq:exoduscopia}. We conclude that
\[
\mathcal{Q}(w,w)=\int_{\pla}\mnh^{-1}\psi^2\,Z(f)^2\,da\geq 0.
\]
Suppose that $Q(w,w)=0$. The fact that $\psi>0$ on $\pla$ implies $Z(f)=0$, i.e., $f$ is constant along the characteristic curves of $\pla$. Since these curves meet the pole and $f\in C_0(\pla)$, then $f=0$ and so $w=0$. Finally, if $\varphi$ is an admissible variation which is $C^3$ off of the pole, then Proposition~\ref{prop:2ndpla} gives us $(A+2\la V)''(0)=Q(w,w)$, where $w:=\escpr{U,N}$ and $U$ is the velocity vector field. From here we conclude that $\pla$ is strictly stable and the proof finishes.
\end{proof}

\begin{remark}
Montefalcone~\cite[Cor.~6.9]{montefalcone2} employed a criterion similar to Corollary~\ref{cor:criteria} (ii) to discuss the strict stability of the regular set of minimal hyperplanes in Carnot groups of step 2.
\end{remark}

\subsection{Surfaces with singular curves}
\label{subsec:cmula}
\noindent

Let $M$ be a $3$-dimensional space form of Webster scalar curvature $\kappa$. Given a CC-geodesic $\Ga:\rr\to M$ of curvature $\mu$ and a number $\la\geq 0$, we define the map $F:\rr^2\to M$ by $F(\eps,s):=\ga_\eps(s)$, where $\ga_\eps$ is the CC-geodesic of curvature $\la$ in $M$ with $\ga_\eps(0)=\Ga(\eps)$ and $\dot{\ga}_\eps(0)=J(\dot{\Ga}(\eps))$. It follows from Lemma~\ref{lem:ccjacobi} (i) that $V_\eps(s):=(\ptl F/\ptl\eps)(\eps,s)$ is a $C^\infty$ vector field with $[\dot{\ga}_\eps,V_\eps]=0$ along $\ga_\eps$. Let $v_\eps:=\escpr{V_\eps,T}$. From Lemma~\ref{lem:ccjacobi} (ii) we deduce that
\begin{equation}
\label{eq:cmulavtheta2}
V_\eps=-(\la\,v_\eps)\,\dot{\ga}_\eps+(v'_\eps/2)\,J(\dot{\ga}_\eps)+v_\eps\,T.
\end{equation}
Clearly $v_\eps(0)=0$ and $v'_\eps(0)=-2$ since $V_\eps(0)=\dot{\Ga}(\eps)$. Moreover, we have
\[
V_\eps'(0)=(2\la)\,\dot{\ga}_\eps(0)+(v_\eps''(0)/2)\,J(\dot{\ga}_\eps(0))
-J(\dot{\ga}_\eps)'(0)-2\,T_{\ga_\eps(0)}.
\]
Note that $J(\dot{\ga}_\eps)'=(2\la)\,\dot{\ga}_\eps-T$ by \eqref{eq:dujv} and \eqref{eq:geoeq}. Therefore
\[
V_\eps'(0)=-(v_\eps''(0)/2)\,\dot{\Ga}(\eps)-T_{\Ga(\eps)}.
\]
On the other hand $V_\eps'=D_{\dot{\ga}_\eps}V_\eps=D_{V_\eps}\dot{\ga}_\eps$ since $[\dot{\ga}_\eps,V_\eps]=0$ along $\ga_\eps$. As a consequence
\[
V_\eps'(0)=D_{V_\eps}J(\dot{\Ga})=J(\dot{\Ga}'(\eps))-T_{\Ga(\eps)}=2\mu\,\dot{\Ga}(\eps)-T_{\Ga(\eps)},
\]
where we have used \eqref{eq:dujv} and that $\dot{\Ga}'=-2\mu\,J(\dot{\Ga})$. The two previous equalities yield $v''_\eps(0)=-4\mu$.

Let $\tau:=4\,(\la^2+\kappa)$. Then Lemma~\ref{lem:ccjacobi} (iii) gives us
\begin{equation}
\label{eq:cmulav}
v_\eps(s)=v(s):=
\begin{cases}
-2\mu s^2-2s,\quad\hspace{5.42cm}\text{ if } \tau=0,
\\
\frac{2}{\sqrt{\tau}}\left\{\frac{-2\mu}{\sqrt{\tau}}\big(1-\cos(\sqrt{\tau} s)\big)-\sin(\sqrt{\tau} s)\right\},\quad\hspace{1.4cm}\text{ if } \tau>0,
\\
\frac{2}{\sqrt{-\tau}}\left\{\frac{-2\mu}{\sqrt{-\tau}}\big(\!\cosh(\sqrt{-\tau} s)-1\big)-\sinh(\sqrt{-\tau} s)\right\},\quad\hspace{0.165cm}\text{ if } \tau<0.
\end{cases}
\end{equation}
The fact that $v'(0)=-2$ implies that $v(s)<0$ for any $s>0$ small enough. If there is a first number $s_0>0$ such that $v(s_0)=0$, then we set $I_0:=[0,s_0]$. In case $v<0$ on $\rr^+$, then we denote $I_0:=[0,+\infty)$. In these conditions, we define
\begin{equation}
\label{eq:sglaGa}
\Sg_\la(\Ga):=F(\rr\times I_0)=\{\ga_\eps(s)\,; \eps\in\rr,\,s\in I_0\}.
\end{equation}

In the next result we establish some properties of $\Sg_\la(\Ga)$. The lemma is known in $\mm(\kappa)$ when $\kappa\geq 0$, see \cite[Prop.~6.3]{rr2} and \cite[Prop.~5.5]{hr1}. Here we follow a general approach which does not use an explicit expression for the CC-geodesics in $M$.

\begin{lemma}
\label{lem:singcurves}
In the previous conditions, we have:
\begin{itemize}
\item[(i)] $\Sg_\la(\Ga)$ is a $C^\infty$ surface immersed in $M$.
\item[(ii)] If $s_0<+\infty$, then the singular set $\Sg_\la(\Ga)_0$ can be parameterized as two CC-geodesics of curvature $\mu$ given by $\Ga(\eps)$ and $\Ga_0(\eps):=F(\eps,s_0)$. Otherwise $\Sg_\la(\Ga)_0=\Ga(\rr)$.
\item[(iii)] If $s_0<+\infty$, then the CC-geodesics $\ga_\eps(s)$ with $s\in I_0$ meet orthogonally $\Ga$ and $\Ga_0$.
\item[(iv)] There is a Riemannian unit normal $N$ such that any CC-geodesic $\ga_\eps(s)$ with $s\in(0,s_0)$ is a characteristic curve of $\Sg_{\la}(\Ga)$. Moreover, the associated mean curvature satisfies $H=\la$.
\end{itemize}
\end{lemma}

\begin{proof}
The map $F(\eps,s):=\ga_\eps(s)$ is $C^\infty$ since $\Ga$ is a $C^\infty$ curve and, for fixed $\la\in\rr$, the solutions of \eqref{eq:geoeq} depends differentiably on the initial data. Note that $(\ptl F/\ptl s)(\eps,s)=\dot{\ga}_\eps(s)$, which is a horizontal vector. Since $(\ptl F/\ptl\eps)(\eps,s)=V_\eps(s)$ and $v_\eps=v<0$ on $(0,s_0)$, we deduce from \eqref{eq:cmulavtheta2} that the restriction of $F$ to $\rr\times(0,s_0)$ is an immersion with empty singular set. On the other hand, we have $(\ptl F/\ptl s)(\eps,0)=J(\dot{\Ga}(\eps))$ and $(\ptl F/\ptl\eps)(\eps,0)=\dot{\Ga}(\eps)$. Hence $F$ is an immersion at any point $(\eps,0)$, and $\Ga(\eps)$ is a curve of singular points. Suppose $s_0<+\infty$. An easy computation from \eqref{eq:cmulav} gives $v'(s_0)=2$. By \eqref{eq:cmulavtheta2} we get $\dot{\Ga}_0(\eps)=V_\eps(s_0)=J(\dot{\ga}_\eps(s_0))$, which is a unit vector. Clearly $\escpr{\dot{\ga}_\eps(s_0),\dot{\Ga}_0(\eps)}=0$. This proves (i) and (iii).

The fact that $\dot{\Ga}_0(\eps)=J(\dot{\ga}_\eps(s_0))$ implies that $\Ga_0(\eps)$ is a horizontal curve parameterized by arc-length. It follows that $\escpr{\dot{\Ga}'_0,\dot{\Ga}_0}=0$ and $\escpr{\dot{\Ga}'_0,T}=0$ along $\Ga_0$. To prove that $\Ga_0$ is a CC-geodesic of curvature $\mu$ it suffices, by \eqref{eq:geoeq}, to see that the function $h_0:=\escpr{\dot{\Ga}'_0,J(\dot{\Ga}_0)}$ equals $-2\mu$ along $\Ga$. By taking covariant derivatives in equality $\escpr{\dot{\Ga}_0,J(\dot{\Ga}_0)}=0$ we get
\[
h_0(\eps)=\escpr{\dot{\Ga}_0(\eps),-J(\dot{\Ga}_0)'(\eps)}=\escpr{J(\dot{\ga}_\eps(s_0)),\frac{d}{d\eps}\bigg|_\eps\dot{\ga}_\eps(s_0)}=\escpr{J(\dot{\ga}_\eps(s_0)),V'_\eps(s_0)},
\]
since $D_{V_\eps}\dot{\ga}_\eps=D_{\dot{\ga}_\eps}V_\eps$ along $\ga_\eps$. Now, we compute $V_\eps'(s_0)$ from \eqref{eq:cmulavtheta2}. Since $v(s_0)=0$ and $v'(s_0)=2$, we obtain
\[
V_\eps'(s_0)=-(2\la)\,\dot{\ga}_\eps(s_0)+(v''(s_0)/2)\,J(\dot{\ga}_\eps(s_0))
+J(\dot{\ga}_\eps)'(s_0)+2\,T_{\ga_\eps(s_0)}.
\]
By substituting into the previous equation we conclude that
\[
h_0(\eps)=\frac{1}{2}\,v''(s_0)=-2\mu,
\]
where the second equality comes from \eqref{eq:cmulav} and the definition of $s_0$. This finishes the proof of (ii).

Now, we define the vector field
\begin{equation}
\label{eq:cmulanormal}
N(\eps,s):=\frac{v(s)\,J(\dot{\ga}_\eps(s))-(v'(s)/2)\,T}{\sqrt{v(s)^2+(v'(s)/2)^2}}.
\end{equation}
Clearly this is a unit normal along $\Sg_{\la}(\Ga)$. Since $v<0$ on $(0,s_0)$ then the associated vector fields $\nuh$ and $Z$ in \eqref{eq:nuh} are given by $
\nuh(\eps,s)=-J(\dot{\ga}_\eps(s))$ and $Z(\eps,s)=\dot{\ga}_\eps(s)$. Having in mind equations \eqref{eq:mc} and \eqref{eq:geoeq} we infer that $2H=\escpr{\dot{\ga}_\eps',\nuh}=-2\la\,\escpr{J(\dot{\ga}_\eps),\nuh}=2\la$. This completes the proof.
\end{proof}

Let $\Ga:\rr\to M$ be a CC-geodesic of curvature $\mu$ in a $3$-dimensional space form. In \cite[Ex.~4.14]{hr2} it is explained how to construct, for any number $\la\geq 0$, a complete, immersed, orientable surface $\cmula(\Ga)$ by matching together surfaces as in \eqref{eq:sglaGa} in a suitable way. A precise description of $\cmula(\Ga)$ in $\mm(\kappa)$ for $\kappa\geq 0$ is found in \cite[Sect.~6]{rr2} and \cite[Sect.~5.2]{hr1}. By Lemma~\ref{lem:singcurves} the surfaces $\cmula(\Ga)$ are $C^\infty$ off of the singular set, which can be parameterized by CC-geodesics of curvature $\mu$. Moreover, by the characterization result in \cite[Thm.~4.5]{hr2} the surfaces $\cmula(\Ga)$ are volume-preserving area-stationary with constant mean curvature $\la$. Indeed, in \cite[Thm.~4.13]{hr2} it is shown that these are the unique complete, connected, orientable, volume-preserving area-stationary $C^2$ surfaces in $M$ whose singular set contains a curve.

Next, we discuss the stability properties of $\cmula(\Ga)$. In general we cannot expect $\cmula(\Ga)$ to be strongly stable. For instance, in \cite[Thm.~5.4]{hrr} it is proved that the minimal helicoids $\mathcal{C}_{0}(\Ga)$ with $\mu\neq 0$ in $\mm(0)$ are not strongly stable. We remark that the variation employed to obtain this result \emph{moves the two singular curves} of $\mathcal{C}_{0}(\Ga)$ . This motivates the following question: is the regular set of $\cmula(\Ga)$ a strongly stable surface? Surprisingly, this question has a positive answer.

\begin{theorem}
\label{th:cmulastability}
Let $M$ be a $3$-dimensional space form. Then, the regular set of any surface $\cmula(\Ga)$ is strictly stable.
\end{theorem}

\begin{proof}
By Corollary~\ref{cor:criteria} (iii) and the definition of $\cmula(\Ga)$, it suffices to check that $\ele(\mnh)<0$ along the regular set of a surface $\Sg_\la(\Ga)$ as in \eqref{eq:sglaGa}.

Consider the unit normal $N$ along $\Sg_\la(\Ga)$ defined in \eqref{eq:cmulanormal}. Following the computations employed to prove Theorem~\ref{th:main1} (v), we can show the equality
\[
v^2\,\ele(\mnh)=2\,v\,v''-(v')^2+\tau\,v^2,
\]
where $\tau:=4\,(\la^2+\kappa)$ and $\kappa$ is the Webster scalar curvature of $M$. By differentiating with respect to $s$ at the right hand side above, and taking into account that $v'''+\tau v'=0$, we get that $2v v''-(v')^2+\tau v^2$ is constant. Since $v(0)=0$ and $v'(0)=2$ we conclude that $v^2\,\ele(\mnh)=-4$ along the regular set of $\Sg_\la(\Ga)$. This completes the proof.
\end{proof}

\begin{remarks}
1. In general, the stability criteria in Remark~\ref{re:trivial} and Corollary~\ref{cor:criteria} (ii) fail for $\cmula(\Ga)$. For instance, the computations in \cite[Sect.~5]{hrr} show that the helicoids $\mathcal{C}_{0}(\Ga)$ in $\mm(0)$ have vertical points and do not satisfy the inequality $|B(Z)+S|^2-4\,\mnh^2\leq 0$.

2. Theorem~\ref{th:cmulastability} illustrates that, as happens for the helicoids $\mathcal{C}_{0}(\Ga)$ in $\mm(0)$, in order to prove instability of $\cmula(\Ga)$ we must employ deformations moving the singular set. As a technical difficulty this requires to compute a second variation formula for such variations.

3. In \cite[Cor.~6.11]{montefalcone2}, Montefalcone used a stability criterion similar to Corollary~\ref{cor:criteria} (ii) to provide an example of a complete minimal hypersurface in the Heisenberg group $\mathbb{H}^n$ with singular set of Hausdorff dimension $n$ and strictly stable regular set. On the other hand, Ritor\'e~\cite{r2} found examples of area-minimizing surfaces with low Euclidean regularity in $\mm(0)$ whose singular set consists of a straight line or several half-lines meeting at a point.
\end{remarks}

\providecommand{\bysame}{\leavevmode\hbox
to3em{\hrulefill}\thinspace}
\providecommand{\MR}{\relax\ifhmode\unskip\space\fi MR }
\providecommand{\MRhref}[2]{%
  \href{http://www.ams.org/mathscinet-getitem?mr=#1}{#2}
} \providecommand{\href}[2]{#2}

\end{document}